\documentclass{amsart}

\usepackage{amsmath}
\usepackage{amssymb}
\usepackage{amsthm}
\usepackage{amsfonts}
\usepackage{calligra}
\usepackage{appendix}
\usepackage{graphicx}
\usepackage{enumerate}
\usepackage{mathrsfs}
\usepackage[english]{babel}
\usepackage{mathtools}

\def\aint{-\nobreak \hskip-8.9pt \nobreak\int}
\def\Aint{-\nobreak \hskip-10.8pt \nobreak\int}
\def\R{\mathbb R}
\def\C{\mathbb C}
\def\N{\mathbb N}

\def\B{\mathbb{B}}

\def\S{{\mathcal S}}

\def\L2{L^2(\mathbb C, \mathbb C)}

\def\H{\mathscr H}

\def\Re{\mathbf{Re}}
\def\A{\mathscr{A}}
\def\SSS{\mathbb S}

\def\E{\mathcal{E}}
\def\tp{\textup}
\def\ext{\operatorname{ext}}
\newcommand{\norm}[1]{\left\lVert#1\right\rVert}
\newcommand{\abs}[1]{\left\lvert#1\right\rvert}

\newcommand{\defeq}{\mathrel{:\mkern-0.25mu=}}
\newcommand{\eqdef}{\mathrel{=\mkern-0.25mu:}}

\newtheorem{theorem}{Theorem}[section]
\newtheorem{lemma}[theorem]{Lemma}
\newtheorem{proposition}[theorem]{Proposition}

\newtheorem{question}[theorem]{Question}
\newtheorem{corollary}[theorem]{Corollary}
\newtheorem{claim}[theorem]{Claim}

\theoremstyle{definition}
\newtheorem{definition}[theorem]{Definition}
\newtheorem{assumption}[theorem]{Assumption}
\newtheorem{example}[theorem]{Example}

\newtheorem{strategy}[theorem]{Strategy}

\theoremstyle{remark}
\newtheorem{remark}[theorem]{Remark}

\numberwithin{equation}{section}

\begin{document}

\title[A note on the Jacobian problem of Coifman, Lions, Meyer and Semmes]{A note on the Jacobian problem of Coifman, Lions, Meyer and Semmes}


\author{Sauli Lindberg}
\address{Department of Mathematics and Statistics, University of Helsinki, P.O. Box 68, 00014
Helsingin yliopisto, Finland}
\curraddr{}
\email{sauli.lindberg@helsinki.fi}
\thanks{The author was supported by the ERC Advanced Grant 834728.}



\date{}

\begin{abstract}
Coifman, Lions, Meyer and Semmes asked in 1993 whether the Jacobian operator and other compensated compactness quantities map their natural domain of definition onto the real-variable Hardy space $\mathcal{H}^1(\R^n)$. We present an axiomatic, Banach space geometric approach to the problem in the case of quadratic operators. We also make progress on the main open case, the Jacobian equation in the plane.
\end{abstract}

\maketitle


\section{Introduction}
The real-variable Hardy space $\mathscr{H}^1(\R^n)$ often acts as a good substitute for $L^1(\R^n)$, mirroring many of the ways in which $\textup{BMO}(\R^n)$ substitutes $L^\infty(\R^n)$~\cite{Gra04,Ste93}. In order to define $\mathscr{H}^1(\R^n)$ we fix $\chi \in C_c^\infty(\R^n)$ with $\int_{\R^n} \chi(x) \, dx \neq 0$, denote $\chi_t(x) \defeq t^{-n} \chi(x/t)$ for every $x \in \R^n$ and $t > 0$ and set
\[\mathcal{H}^1(\R^n) \defeq \left\{ f \in \mathcal{S}'(\R^n) \colon \sup_{t > 0} \abs{f * \chi_t(\cdot)} \in L^1(\R^n) \right\}.\]
We endow $\mathscr{H}^1(\R^n)$ with the norm $\norm{f}_{\mathscr{H}^1} \defeq \norm{\sup_{t > 0} \abs{f * \chi_t(\cdot)}}_{L^1}$.

Connections between $\mathscr{H}^1$ integrability, commutators and weak sequential continuity were explored by Coifman \& al. in the highly influential work ~\cite{CLMS93}. Coifman \& al. showed that when $n \ge 2$, the Jacobian determinants of mappings in $\dot{W}^{1,n}(\R^n,\R^n)$ and several other compensated compactness quantities belong to $\mathscr{H}^1(\R^n)$. The result was motivated by M\"{u}ller's higher integrability result on Jacobians: if $u \in W^{1,n}(\Omega,\R^n)$ and $J_u \ge 0$ a.e. in an open set $\Omega \subset\R^n$, then $J u \log(2 + J u) \in L^1_{loc}(\Omega)$~\cite{Mul89}, in direct analogy to Stein's classical $L \tp{log} L$ result on $\mathscr{H}^1$ functions. For later developments of the $\mathcal{H}^1$ theory of compensated compactness quantities see e.g.~\cite{CG92,Gra92,GKL20,GR22,GRS22,IO02,LS20,LMZW97,Lin17,MS18,PW00,Str01,Wu96}.

Coifman \& al. proceeded to ask whether these nonlinear quantities are surjections onto $\mathcal{H}^1(\R^n)$~\cite[p. 258]{CLMS93}. The most famous open case is the following:
\begin{equation} \label{e:Jacobian problem}
\tp{Is } J \colon \dot{W}^{1,n}(\R^n,\R^n) \to \H^1(\R^n) \tp { surjective?}
\end{equation}
As a partial result, Coifman \& al. showed that $\mathscr{H}^1(\R^n)$ is the smallest Banach space that contains the range of the Jacobian. More precisely,
\begin{equation} \label{e:Jacobian decomposition}
\mathscr{H}^1(\R^n) = \left\{\sum_{j=1}^\infty J u_j \colon \sum_{j=1}^\infty \norm{D u_j}_{L^n}^n < \infty\right\}
\end{equation}
\cite[Theorem III.2]{CLMS93}. Further partial results were presented in~\cite{GKL20,GKL21B,Lin15}. When the domain of definition of $J$ is the \emph{inhomogeneous} Sobolev space $W^{1,n}(\R^n,\R^n)$, the author proved non-surjectivity in ~\cite{Lin17}.

In bounded domains, the Dirichlet problem $Ju = f$ in $\Omega$, $u = \tp{id}$ on $\partial \Omega$ has a classical theory starting from the seminal works of Moser and Dacorogna in~\cite{DM90,Mos65} and reviewed in~\cite{CDK12}. In a setting close to ours, when $u \in W^{1,n}_{\tp{id}}(\Omega,\R^n)$ has Jacobian $Ju \geq 0$ a.e., M\"uller's higher integrability result implies that $Ju \in L \log L(\Omega)$. As an analogue of \eqref{e:Jacobian problem}, Hogan \& al. asked whether every non-negative $f \in L \log L(\Omega)$ with mean 1 has a solution $u \in W^{1,n}_{\tp{id}}(\Omega,\R^n)$~\cite[p. 206]{HLMZ00}. Counter-examples were recently given by Guerra \& al. in~\cite{GKL21A}. For all $L^{q/n}$ data, $1 < q < n$, "pointwise" (as opposed to distributional) solutions in $W^{1,q}(\Omega,\R^n)$ with arbitrary boundary values in $W^{1-1/q,q}(\Omega,\R^n)$ were constructed in~\cite{KRW15}. At any rate, compared to the Dirichlet problem, different ideas are needed in the case of $\R^n$ due to the unboundedness of the domain and the absence of boundary conditions. 

\vspace{0.2cm}
In \cite{Iwa97}, Iwaniec conjectured that for every $n \ge 2$ and $p \in [1,\infty)$ the Jacobian operator $J \colon \dot{W}^{1,np}(\R^n,\R^n) \to \mathcal{H}^p(\R^n)$ not only is surjective but also has a continuous right inverse $G \colon \mathscr{H}^p(\R^n) \to \dot{W}^{1,np}(\R^n,\R^n)$; recall that $\mathscr{H}^p(\R^n) = L^p(\R^n)$ whenever $1 < p < \infty$. Hyt\"onen proved in~\cite{Hyt21} the natural analogue of \eqref{e:Jacobian decomposition}, so that $L^p(\R^n)$ is, again, the minimal Banach space containing $J(\dot{W}^{1,np}(\R^n,\R^n))$. We discuss Hyt\"onen's contribution in \textsection \ref{s:The Jacobian equation with Lp data}.

We next briefly summarise some of the ideas presented in \cite{Iwa97}. Whenever $f \in \mathcal{H}^p(\R^n)$ and the equation $J v = f$ has a solution, it has a \emph{minimum norm solution} $u \in \dot{W}^{1,np}(\R^n,\R^n)$, that is, $J u = f$ and $\int_{\R^n} \abs{D u}^{np} = \min_{J v = f} \int_{\R^n} \abs{D v}^{np}$. Furthermore, the range of $J$ is dense in $\mathcal{H}^p(\R^n)$. Iwaniec suggested a possible way of finding a continuous right inverse $G$:

\begin{strategy} \label{s:Iwaniec}
When $n \geq 2$ and $1 \leq p < \infty$, the following claims would yield a continuous right inverse of $J \colon \dot{W}^{1,np}(\R^n,\R^n) \to \mathscr{H}^p(\R^n)$:
\begin{enumerate}
\item Every energy-minimal solution $u$ satisfies $\norm{D u}_{L^{np}}^{n} \lesssim \norm{J u}_{\mathscr{H}^p}$.
\label{i:Iwaniec (1)}

\item For every $f \in \mathscr{H}^p(\R^n)$ there is a unique energy-minimal solution $u_f$, modulo rotations. \label{i:Iwaniec (2)}

\item There exist rotations $R_f \in SO(n)$ such that $f \mapsto R_f u_f$ is a continuous right inverse of $J$. \label{i:Iwaniec (3)}
\end{enumerate}
\end{strategy}

Claim \eqref{i:Iwaniec (1)} would easily imply the  surjectivity of $J \colon \dot{W}^{1,np}(\R^n,\R^n) \to \mathscr{H}^p(\R^n)$ (see~\cite[p. 36]{Lin15}). In~\cite{GKL20}, \eqref{i:Iwaniec (1)} was shown to be, in fact, \emph{equivalent} to the surjectivity of $J \colon \dot{W}^{1,np}(\R^n,\R^n) \to \mathscr{H}^p(\R^n)$; this amounts to an open mapping theorem for the Jacobian. In~\cite{GKL21B}, in turn, \eqref{i:Iwaniec (2)} was shown to be false whenever $n = 2$ and $1 \leq p < \infty$. Nevertheless, in~\cite{GKL21B}, Guerra \& al. found an explicit class of data whose energy-minimal solution is, indeed, unique up to rotations. Another large class of such data is constructed in Theorem \ref{t:Theorem on uniqueness of minimum norm solutions} below. Despite the falsity of \eqref{i:Iwaniec (2)}, claim \eqref{i:Iwaniec (1)} and Iwaniec's conjecture itself remain open. In the case $n = 2$, $p=1$, a novel Banach space geometric approach was presented in~\cite{Lin15} to attack claim \eqref{i:Iwaniec (1)}.

In this work, we present a natural abstract framework for the ideas of \cite{Lin15} and study the surjectivity question for rather general quadratic compensated compactness quantities, streamlining the exposition of~\cite{Lin15} considerably. We also make further progress on the main special case $J \colon \dot{W}^{1,2}(\R^2,\R^2) \to \mathcal{H}^1(\R^2)$. In \textsection \ref{s:The Jacobian equation with Lp data}, we discuss how to adapt the methods to the case $n = 2$, $1 < p < \infty$.

\subsection{Connection to commutators} \label{ss:Connection to commutators}
Question \eqref{e:Jacobian problem} can be seen as a variant of the classical \emph{factorisation problem} from complex analysis~\cite[\textsection 4.2]{Rud69}. Indeed, an equivalent formulation of \eqref{e:Jacobian problem} in terms of differential forms is whether $\H^1(\R^n) = * \wedge_{j=1}^n \tp{d} \dot{W}^{1,n}(\R^n)$ (since $Ju = * \tp{d} u_1 \wedge \cdots \wedge \tp{d} u_n$). The decomposition \eqref{e:Jacobian decomposition} is called a \emph{weak factorisation} of $\H^1(\R^n)$. In the plane, one can also deduce \eqref{e:Jacobian decomposition} from the two-sided commutator estimate
\begin{equation} \label{Two-sided estimate}
c \norm{b}_{\operatorname{BMO}} \le \norm{[b,T]}_{L^2 \to L^2} \le C \norm{b}_{\operatorname{BMO}}
\end{equation}
for the Beurling transform $T = \S \colon L^2(\C,\C) \to L^2(\C,\C)$ and $b \in \tp{BMO}(\C)$~\cite{Lin15}. 

Estimates of the form \eqref{Two-sided estimate} go back to the seminal work~\cite{CRW76} of Coifman \& al., with Nehari's theorem on Hankel operators as a precursor. When $j \in \{1,\ldots,n\}$ and $T = R_j$ is a Riesz transform, the \emph{upper bound estimate} $\norm{[b,R_j]}_{L^2 \to L^2} \le C \norm{b}_{\text{BMO}}$~\cite{CRW76}, the formula
\begin{equation} \label{e:Commutator formula}
\int_{\R^n} b (\omega R_j \gamma + \gamma R_j \omega) = \int_{\R^n} \omega [b,R_j] \gamma
\end{equation}
and $\mathcal{H}^1-\operatorname{BMO}$ duality imply that $\omega R_j \gamma + \gamma R_j \omega \in \H^1(\R^n)$ for all $\omega,\gamma \in L^2(\R^n)$. The \emph{lower bound estimate} $\norm{[b,R_j]}_{L^2 \to L^2} \ge c \norm{b}_{\text{BMO}}$~\cite{CRW76,Jan78,Uch78} gives the weak factorisation $\mathscr{H}^1(\R^n) = \{\sum_{i=1}^\infty (\omega_i R_j \gamma_i + \gamma_i R_j \omega_i) \colon \sum_{i=1}^\infty (\norm{\omega_i}_{L^2}^2 + \norm{\gamma_i}_{L^2}^2) < \infty\}$ by simple functional analysis. By a result of Uchiyama, the commutator $[b,R_j]$ is compact if and only if $b \in \tp{CMO}(\R^n)$. Note that \eqref{e:Commutator formula} and the compactness of $[b,R_j]$ immediately lead to the weak-to-weak$^*$ sequential continuity of the quadratic operator $(\omega,\gamma) \mapsto \omega R_j \gamma + \gamma R_j \omega \colon L^2(\R^n)^2 \to \H^1(\R^n)$. These results have been extended in numerous ways, as reviewed in~\cite{Wick20}. In \textsection \ref{ss:Commutators of Calderon-Zygmund operators} we briefly discuss the case of commutators with Calder\'on-Zygmund operators.

In Assumptions \ref{Assumption on the bilinear operator}--\ref{Assumption on the bilinear operator 2} below, we axiomatise the above-mentioned properties of the Riesz transform $R_j$ and the quantity $\omega R_j \gamma + \gamma R_j \omega$. We study the factorisation problem for quadratic weakly continuous quantities in real-variable function spaces such as $\H^1(\R^n)$, using isometric Banach space geometry as the main tool. Since our choice of methodology is rather unconventional in the study of nonlinear PDE's, we motivate it at length in \textsection \ref{ss:Overall strategy and aim}. First, however, we specify the mathematical setting of this paper.

\subsection{The mathematical setting}
We fix a real Banach space $X$ with a separable dual $X^*$ and a real or complex Hilbert space $H$, and we denote the coefficient field of $H$ by $\mathbb{K} \in \{\R,\C\}$. The canonical examples are $X = \operatorname{CMO}(\R^n)$, $X^* = \mathcal{H}^1(\R^n)$, $X^{**} = \operatorname{BMO}(\R^n)$ and $H = L^2(\R^n,\mathbb{R}^m)$ or $H = L^2(\R^n,\C)$, where $n ,m \in \N$.

\begin{assumption} \label{Assumption on the bilinear operator}
A bilinear mapping
\[(b,\omega) \mapsto T_b \omega \colon X^{**} \times H \to H\]
satisfies the following conditions:

\renewcommand{\labelenumi}{(\roman{enumi})}
\begin{enumerate}
\item $c \norm{b}_{X^{**}} \le \norm{T_b}_{H \to H} \le C \|b\|_{X^{**}}$ for every $b \in X^{**}$,

\item $T_b$ is compact for every $b \in X$.
\end{enumerate}
\end{assumption}

\begin{assumption} \label{Assumption on the bilinear operator 2}
The bilinear mapping $(b,\omega) \mapsto T_b \omega$ satisfies the following conditions for every $b \in X^{**}$:

\renewcommand{\labelenumi}{(\roman{enumi})}
\begin{enumerate}
\item $T_b$ is self-adjoint.

\item $\norm{T_b}_{H \to H} = \sup_{\norm{f}_H=1} \langle T_b f, f \rangle$.
\end{enumerate}
\end{assumption}

\begin{definition} \label{Definition of the quadratic quantity}
Given $X$, $H$ and $(b,\omega) \mapsto T_b f$ we define a norm-to-norm and weak-to-weak$^*$ sequentially continuous mapping $Q \colon H \to X^*$ by
\[\langle b, Q \omega \rangle_{X-X^*} \defeq \langle T_b \omega, \omega \rangle_H.\]
\end{definition}

Henceforth, Assumptions \ref{Assumption on the bilinear operator}--\ref{Assumption on the bilinear operator 2} will remain in place for the rest of the introduction. Assumption \ref{Assumption on the bilinear operator 2} is made mainly to make the quadratic operator $Q$ real-valued and ensure that $X^* = \{\sum_{j=1}^\infty Q \omega_j \colon \sum_{j=1}^\infty \norm{\omega_j}_H^2 < \infty\}$. Whenever the map $(b,\omega) \mapsto T_b \omega$ satisfies Assumption \ref{Assumption on the bilinear operator}, the modified operator $(b,(\omega,\gamma)) \mapsto \tilde{T}_b (\omega,\gamma) \defeq (T_b^* \gamma, T_b \omega) \colon X^{**} \times (H \times H) \to H \times H$ satisfies Assumptions \ref{Assumption on the bilinear operator}--\ref{Assumption on the bilinear operator 2} (see Proposition \ref{p:Corollary on modified operator}). Examples \ref{ex:Complex Jacobian}--\ref{ex:Operators in terms of Hilbert transforms} illustrate the role of Assumption \ref{Assumption on the bilinear operator 2} further.

The planar Jacobian arises as follows (see Example \ref{ex:Complex Jacobian}): defining $T_b \colon L^2(\C,\C) \to L^2(\C,\C)$ by $T_b \omega \defeq \overline{(\S b - b \S) \overline{\S \omega}}$, where $\S$ is the Beurling transform, we can write $Q \omega = \abs{\S \omega}^2 - \abs{\omega}^2 = \abs{u_z}^2 - \abs{u_{\bar{z}}}^2 = Ju$, where $u \in \dot{W}^{1,2}(\C,\C)$ is the Cauchy transform of $\omega \in L^2(\C,\C)$.

\vspace{0.3cm}
Assumptions \ref{Assumption on the bilinear operator}--\ref{Assumption on the bilinear operator 2} are not enough to determine whether $Q(H) = X^*$. Indeed, $Q \colon L^2(\R) \to \H^1(\R)$, $Q(\omega) = \omega^2 - (H \omega)^2$ is non-surjective but (its G\^{a}teaux derivative) $\tilde{Q} \colon L^2(\R) \times L^2(\R) \to \H^1(\R)$, $\tilde{Q}(\omega,\gamma) = \omega \gamma - H\omega H \gamma$ is surjective (see Example \ref{ex:Operators in terms of Hilbert transforms}). We address the following question:
\begin{question} \label{q:Assumption question}
Under which extra assumptions is $Q(H) = X^*$?
\end{question}

Below we mostly study all the operators given by Definition \ref{Definition of the quadratic quantity} in a unified manner, as precise information on $Q(H)$ is valuable whether one intends to prove surjectivity or non-surjectivity. However, in Assumption \ref{Assumption 2} we specify a useful criterion which \emph{holds} for the operators $\omega \mapsto \omega^2-(H\omega)^2 \colon L^2(\R) \to \H^1(\R)$ and $\omega \mapsto \abs{\S\omega}^2-\abs{\omega}^2 \colon L^2(\C,\C) \to \H^1(\C)$ but \emph{fails} for their G\^ateaux derivatives.

\subsection{Overall strategy and aim} \label{ss:Overall strategy and aim}
We next describe the motivation behind the approach adopted in ~\cite{Lin15} and this paper. Elementary proofs of various statements are given in \textsection \ref{Banach space geometric preliminaries}.

\vspace{0.3cm}
$\bullet$ Whenever $f \in X^*$ and the equation $Q \gamma = f$ has a solution, it has a \emph{minimum norm solution} $\omega \in H$, that is, $Q \omega = f$ and $\norm{\omega}_H = \min_{Q \gamma = f} \norm{\gamma}_H$.

$\bullet$ We define an \emph{energy functional} $\E \colon X^* \to \R \cup \{\infty\}$ by
\[\E(f) \defeq \begin{cases}
\min\{\norm{\omega}_H^2 \colon Q \omega = f\}, & f \in Q(H), \\
\infty, & f \notin Q(H).
\end{cases}\]
Almost by definition, $\E \ge \norm{\cdot}_{X^*}$. Claim \eqref{i:Iwaniec (1)} in Strategy \ref{s:Iwaniec} can be stated equivalently as the estimate $\E \lesssim \norm{\cdot}_{X^*}$.

$\bullet$ By a nonlinear version of the Banach-Schauder open mapping theorem, recently proved by Guerra \& al. in~\cite{GKL20}, for translation-invariant operators (such as the Jacobian) we have
\[\tp{either } \E \lesssim \norm{\cdot}_{X^*} \tp{ in } X^* \tp{ or } \E = \infty \tp{ outside a meagre set.}\]
In particular, the surjectivity of $Q \colon H \to X^*$ is \emph{equivalent} to the statement that $Q(H)$ is dense in $X^*$ and every minimum norm solution satisfies $\norm{\omega}_H^2 \lesssim \norm{Q \omega}_{X^*}$. This gives a metamatematical justification for taking claim \eqref{i:Iwaniec (1)} as a goal.

$\bullet$ Given a minimum norm solution $\omega \in H$ of $Q \omega = f$, we may use calculus of variations to study $\omega$ via perturbed solutions $\omega_\epsilon \in H$ of $Q \omega_\epsilon = f$, $\norm{\omega_\epsilon - \omega}_H \to 0$. However, it tends to be very hard to construct variations that satisfy the nonlinear constraint $Q \omega_\epsilon = f$. In particular, the first variation $\omega_\epsilon = \omega + \epsilon \varphi$, $\varphi \in H$, is in general unavailable.

$\bullet$ The constraint $Q \omega_\epsilon = f$ could be relaxed if we found a \emph{Lagrange multiplier}, that is, $b \in X^{**}$ satisfying
\begin{equation} \label{Lagrange multiplier condition}
\left. \frac{d}{d\epsilon} (\langle b, Q(f + \epsilon \varphi) \rangle_{X^{**}-X^*} - \norm{f + \epsilon \varphi}_H^2) \right|_{\epsilon = 0} = 0 \qquad \text{for every } \varphi \in H.
\end{equation}
If, furthermore, $\|b\|_{X^{**}} \le C$ uniformly in $f$, then setting $\varphi = f$ in \eqref{Lagrange multiplier condition} would give the sought inequality $\E \lesssim \norm{\cdot}_{X^*}$. However, given an arbitrary minimum norm solution, the construction of a Lagrange multiplier $b \in X^{**}$ is a formidable task--in particular, the standard Liusternik-Shnirelman method is not applicable.

$\bullet$ Nevertheless, many minimum norm solutions do possess a Lagrange multiplier. Indeed, \eqref{Lagrange multiplier condition} says that $\omega$ is a critical point of the functional $I_b \colon H \to \R$,
\[I_b(\theta) \defeq \langle b, Q \theta \rangle_{X^{**}-X^*} - \norm{\theta}_H^2 = \langle T_b \theta, \theta \rangle_{H} - \norm{\theta}_H^2.\]
Now, if $b \in X$ and $\norm{T_b}_{H \to H} = 1$, then $\sup_{\norm{\theta}_H=1} I_b(\theta) = 0$ is attained at some $\omega \in \mathbb{S}_H$, so that $b$ is a Lagrange multiplier of $\omega$ and $\norm{\omega}_H^2 \lesssim \norm{Q \omega}_{X^*}$.

$\bullet$ As every $b \in X$ with $\norm{T_b}_{H \to H} = 1$ is a Lagrange multiplier, we use the norm of $X$ given by
\[\norm{b}_{X_Q} \defeq \sup_{\norm{\omega}_H = 1} \langle b, Q \omega \rangle_{X-X^*} = \norm{T_b}_{H \to H},\]
endow $X^*$ with the dual norm
\[\norm{f}_{X^*_Q} \defeq \sup_{\norm{b}_{X_Q}=1} \langle f,b \rangle_{X^*-X}\]
and denote the resulting Banach spaces by $X_Q$ and $X^*_Q$. Thus every $b \in \mathbb{S}_{X_Q}$ is a Lagrange multiplier of some $\omega \in \mathbb{S}_H$ with $Q \omega \in \mathbb{S}_{X_Q^*}$.

$\bullet$
As a consequence, the set
\[\mathscr{A} \defeq \{\omega \in \mathbb{S}_H \colon Q \omega \in \mathbb{S}_{X^*_Q}\} \subset \mathbb{S}_H\]
and its image $Q(\mathscr{A}) \subset \mathbb{S}_{X^*_Q}$ are rather large. The motivation above leads to the following more precise variant of Question \ref{q:Assumption question}, presented for the planar Jacobian in~\cite{Lin15}:

\begin{question} \label{q:Lindberg}
Is $Q(\mathscr{A}) = \mathbb{S}_{X^*_Q}$? Equivalently, is $\E = \norm{\cdot}_{X^*}$?
\end{question}

As already noted, there exist some natural cases where the answer to Question \ref{q:Lindberg} is negative. A positive answer holds for the simple operator $(b,(\omega,\gamma)) \mapsto (b\omega,-b\gamma) \colon \R \times \R^2 \to \R^2$. It is hoped that suitable further properties of $T$, beyond Assumptions \ref{Assumption on the bilinear operator}--\ref{Assumption on the bilinear operator 2}, will yield an answer to Question \ref{q:Lindberg} for natural classes of operators. In this paper we present several partial results.

\subsection{Main results}
As a first largeness criterion on $Q(\A)$ we mention that
\begin{equation} \label{Krein-Milman set inclusion}
Q(\mathscr{A}) \supset \operatorname{ext}(\mathbb{B}_{X^*_Q}),
\end{equation}
where $\operatorname{ext}(\mathbb{B}_{X^*_Q})$ is the set of extreme points of $\mathbb{B}_{X_Q^*}$. Since $X^*_Q$ is a separable dual space, the Bessaga-Pelczynski Theorem (Theorem \ref{Bessaga-Pelczynski theorem}) implies that $\overline{\operatorname{co}}(\operatorname{ext}(\mathbb{B}_{X^*_Q})) = \mathbb{B}_{X^*_Q}$. The inclusion \eqref{Krein-Milman set inclusion} is contained in Theorem \ref{l:Krein-Milman result}. We will find more refined information on $Q(\mathscr{A})$ by using the set-valued duality mapping.

\begin{definition}
The \emph{duality mapping} $D \colon \mathbb{S}_{X_Q} \to 2^{\mathbb{S}_{X^*_Q}}$ is defined by
\[D(b) \defeq \{h \in \mathbb{S}_{X^*_Q} \colon \langle b, h \rangle_{X_Q-X_Q^*} = 1\}.\]
We denote the set of \emph{norm-attaining points} by $\textit{NA}_{\norm{\cdot}_{X_Q}} \defeq \cup_{b \in \mathbb{S}_{X_Q}} D(b)$.
\end{definition}

By the Bishop-Phelps Theorem \cite[Theorem 3.54]{FHHMPZ}, $NA_{\norm{\cdot}_{X_Q}}$ is dense in $\mathbb{S}_{X^*_Q}$. Since $Q(\mathscr{A})$ is closed, Question \ref{q:Lindberg} reduces to the question whether $D(b) \cap Q \mathscr{A} = D(b)$ for every $b \in \mathbb{S}_{X_Q}$. Before presenting partial results we formulate a useful extra assumption which is satisfied by the planar Jacobian. Its aim is to quantify the symmetries of the class $\A$ (see Remark \ref{r:Symmetries}).

\begin{assumption} \label{Assumption 2}
If $\omega,\gamma \in \mathscr{A}$ satisfy $Q \omega = Q \gamma$, then $Q'_\omega (c \gamma) \neq 0$ for some $c \in \mathbb{S}_{\mathbb{K}}$.
\end{assumption}

The following result collects partial results on Question \ref{q:Lindberg}; for relevant definitions see \textsection \ref{Smoothness properties of norms and duality mappings}. The parts (iv)--(vi) are new also in the case of the Jacobian.

\begin{theorem} \label{t:Theorem on duality mapping}
Under Assumptions \ref{Assumption on the bilinear operator}--\ref{Assumption on the bilinear operator 2}, the following statements hold:
\renewcommand{\labelenumi}{(\roman{enumi})}
\begin{enumerate}
\item For every $b \in \mathbb{S}_{X_Q}$, the convex set $D(b)$ has finite affine dimension.

\item For every $b \in \mathbb{S}_{X_Q}$, $D(b) \cap Q(\mathscr{A})$ contains $\operatorname{ext}(D(b))$ and is path-connected.

\item The norm and relative weak$^*$ topologies coincide in $Q (\mathscr{A})$.
\end{enumerate}
Under the further Assumption \ref{Assumption 2}, the following statements hold:
\begin{enumerate}
\setcounter{enumi}{3}
\item For every $b$ in a dense, relatively open subset of $\mathbb{S}_{X_Q}$, $D(b) = \{Q \omega\}$ for some $\omega \in \mathscr{A}$ and $\norm{\cdot}_{X_Q}$ is Fr\'echet differentiable at $b$.

\item $D \colon \mathbb{S}_{X_Q} \to 2^{\mathbb S_{X_Q^*}}$ is a cusco map.

\item The norm and relative weak$^*$ topologies also coincide in $\textit{NA}_{\norm{\cdot}_{X_Q}}$.
\end{enumerate}
\end{theorem}

The proof is presented in \textsection \ref{s:Proof of Theorem on duality mapping}. The second main result concerns uniqueness of minimum norm solutions and is proved in  \textsection \ref{s:The proof of Theorem on uniqueness of minimum norm solutions}.

\begin{theorem} \label{t:Theorem on uniqueness of minimum norm solutions}
Suppose Assumptions \ref{Assumption on the bilinear operator}--\ref{Assumption on the bilinear operator 2} and \ref{Assumption 2} hold, and let $f \in \operatorname{ext}(\mathbb{S}_{X_Q^*})$. Then the minimum norm solution $\omega \in \A$ of $Q \omega = f$ is unique up to multiplication by $c \in \mathbb{S}_{\mathbb{K}}$.
\end{theorem}

In view of \eqref{Krein-Milman set inclusion} it is natural to ask whether $X_Q^*$ is strictly convex, that is, $\tp{ext}(\mathbb{B}_{X_Q^*}) = \mathbb{S}_{X_Q^*}$. In the case of the planar Jacobian, the answer is negative, as an immediate consequence of Theorems \ref{t:Theorem on duality mapping}--\ref{t:Theorem on uniqueness of minimum norm solutions} and the non-uniqueness of general minimum norm solutions:

\begin{corollary} \label{c:Not stricly convex}
In the case of $\omega \mapsto Q \omega = \abs{\S\omega}^2-\abs{\omega}^2 \colon L^2(\C,\C) \to \mathcal{H}^1(\C)$ we have $\tp{ext}(\mathbb{B}_{X_Q^*}) \subsetneq \mathbb{S}_{X_Q^*}$.
\end{corollary}

It is unclear to the author whether an analogue of Corollary \ref{c:Not stricly convex} holds for the G\^ateaux derivative $(\omega, \gamma) \mapsto Q'_\omega \gamma \colon L^2(\C,\C)^2 \to \H^1(\C)$.

\section{Preliminaries}
The main tools of this work come from isometric Banach space geometry. We collect some notions and results and refer to ~\cite{FHHMZ} and ~\cite{FHHMPZ} for most of the proofs.

\subsection{Smoothness properties of norms and duality mappings} \label{Smoothness properties of norms and duality mappings}
In this subsection, $X$ is a real Banach space. Suppose $U \subset Z$ is open and $g \colon U \to Z$ is convex, where $Z$ is a real Banach space. We say that $g$ is \emph{G\^{a}teaux differentiable} at $x \in U$ if there exists $L \in \mathscr{B}(X,Z)$ such that
\begin{equation} \label{Gateaux condition}
\lim_{t \to 0} \norm{\frac{g(x+th)-g(x)}{t} -Lh} = 0 \quad \tp{for every } h \in X.
\end{equation}
The operator $L$ is then denoted by $g'_x$ and called the \emph{G\^ateaux derivative} of $g$ at $x$. If the limit in \eqref{Gateaux condition} is uniform in $h \in \mathbb{S}_X$, then $g$ is said to be \emph{Fr\'{e}chet differentiable} at $x$.

Note that the convex function $g(x) = \norm{x}_X$ is G\^{a}teaux (Fr\'{e}chet) differentiable at $x \in \mathbb{S}_X$ if and only if it is G\^{a}teaux (Fr\'{e}chet) differentiable at $\lambda x$ for every $\lambda \in \R \setminus \{0\}$. We will use the following theorem of Asplund and Lindenstrauss (see ~\cite[Theorem 8.21]{FHHMPZ}) on the norm of $X$.

\begin{theorem} \label{Asplund-Lindenstrauss theorem}
Suppose $X^*$ is separable and $g \colon X \to \R$ is a continuous convex function. Then $g$ is Fr\'{e}chet differentiable in a dense $G_\delta$ subset of $X$.
\end{theorem}

Recall that the duality mapping $D \colon \mathbb{S}_X \to 2^{\mathbb{S}_{X^*}}$ is defined by
\[D(x) \defeq \{x^* \in \mathbb{S}_{X^*} \colon \langle x^*, x \rangle_{X^*-X} = 1\}.\]
If $g(x) = \norm{x}_X$ is G\^ateaux differentiable at $x \in \SSS_X$, then $D(x) = \{g'_x\}$.

\begin{definition}
Suppose $Y$ and $Z$ are Hausdorff topological spaces. A mapping $F \colon Y \to 2^Z \setminus \{\emptyset\}$ is said to be \emph{upper semicontinuous} at $y \in Y$ if for any open set $V \supset F(y)$, there exists a neighbourhood $U$ of $y$ in $Y$ such that $F(U) \subset V$.
\end{definition}

We recall characterisations of the norm-to-norm upper semicontinuity of the duality mapping~\cite{GGS78, HL92}.

\begin{theorem} \label{t:GGS}
Let $x \in \mathbb{S}_X$. The following statements are equivalent.
\renewcommand{\labelenumi}{(\roman{enumi})}
\begin{enumerate}
\item $D$ is norm-to-norm upper semicontinuous at $x$ and $D(x)$ is norm compact.

\item For every sequence of points $x_j^*$ in $\mathbb{S}_{X^*}$ such that $\langle x_j^*, x \rangle \to 1$, there exists a subsequence convergent to some $x^* \in D(x)$.

\item The weak$^*$ and norm topologies agree on $\mathbb{S}_{X^*}$ at points of $D(x)$.
\end{enumerate}
\end{theorem}

The norm-to-norm upper semicontinuity of $D$ can also be used to characterise Fr\'{e}chet differentiability of the norm (see ~\cite[Corollary 7.16]{FHHMZ}).

\begin{theorem}
The norm $\norm{\cdot}_X$ is Fr\'{e}chet differentiable at $x \in \mathbb{S}_X$ if and only if $D(x)$ is a singleton and $D$ is norm-to-norm upper semicontinuous at $x$.
\end{theorem}

\begin{definition}
Let $Y$ and $Z$ be Hausdorff topological spaces. A set-valued map $F \colon Y \to 2^Z \setminus \{\emptyset\}$ is said to be \emph{cusco} (convex upper semicontinuous nonempty compact-valued) if it is norm-to-norm upper semicontinuous and $F(y)$ is compact for every $y \in Y$.
\end{definition}

\subsection{Extreme points of the unit ball of a separable dual}
When $C$ is a convex subset of a real Banach space $Y$, a point $y \in C$ is an \emph{extreme point} of $C$ if there exists no proper line segment that contains $y$ and lies in $C$.

\begin{definition}
A Banach space is said to have the \emph{Krein-Milman property} if every closed, bounded, convex set is the closed convex hull of its extreme points.
\end{definition}

We recall a result of C. Bessaga and A. Pe{\l}czynski (see e.g.~\cite[p. 198]{DU77}).

\begin{theorem} \label{Bessaga-Pelczynski theorem}
Every separable dual space has the Krein-Milman property.
\end{theorem}

Recall that in Assumptions \ref{Assumption on the bilinear operator}--\ref{Assumption on the bilinear operator 2}, the dual space $X^*$ is separable. The real-variable Hardy space $\mathcal{H}^1(\R^n) = (\operatorname{CMO}(\R^n))^*$ satisfies this criterion (see ~\cite[Proposition 2.15]{Lin15}). We also recall Milman's theorem on extreme points which is formulated as follows (see~\cite[Theorem 3.66]{FHHMZ}).

\begin{theorem} \label{Milman's theorem}
Let $K$ be a non-empty subset of a Hausdorff locally convex space such that $\overline{\operatorname{co}}(K)$ is compact. Then every extreme point of $\overline{\operatorname{co}}(K)$ lies in $\overline{K}$.
\end{theorem}

\subsection{Self-adjoint variants of non-self-adjoint operators} \label{Self-adjoint variants of non-self-adjoint operators}
When Assumption \ref{Assumption on the bilinear operator} holds but Assumption \ref{Assumption on the bilinear operator 2} does not, we use a natural self-adjoint modification of $T_b$. Its existence and uniqueness are guaranteed by the following simple lemma.

\begin{lemma} \label{l:Lemma on the self-adjoint modification}
Suppose $A \colon H \to H$ is $\mathbb{K}$-linear. Then there exists a unique self-adjoint $\mathbb{K}$-linear operator $B \colon H \times H \to H \times H$ such that
\[\langle B(x,y), (x,y) \rangle_{H \times H} = 2 \mathbf{Re} \langle A x, y \rangle_H \quad \text{for all } x,y \in H.\]
The operator $B$ is of the form $B(x,y) = (A^* y, A x)$ and satisfies $\norm{B}_{H \times H \to H \times H} = \norm{A}_{H \to H}$.
\end{lemma}

\begin{proposition} \label{p:Corollary on modified operator}
If a bilinear mapping $T \colon X^{**} \times H \to H$ satisfies Assumption \ref{Assumption on the bilinear operator}, then the modified operator
\[(b,f,g) \mapsto \widetilde{T_b}(f,g) \colon X^{**} \times (H \times H) \to H \times H, \qquad \widetilde{T_b}(f,g) \defeq (T_b^* g, T_b f)\]
satisfies Assumptions \ref{Assumption on the bilinear operator}--\ref{Assumption on the bilinear operator 2}.

If $T$ satisfies Assumptions \ref{Assumption on the bilinear operator}--\ref{Assumption on the bilinear operator 2}, then $\tilde{T}$ is related to the G\^{a}teaux derivative of $Q$ by
\begin{equation} \label{Gateaux derivative formula}
\langle \tilde{T}_b(f,g),(f,g) \rangle_{H \times H} = \langle b, Q'_f g \rangle_{X^{**}-X^*} \qquad \text{for all } f,g \in H, \, b \in X^{**}.
\end{equation}
\end{proposition}

We illustrate Proposition \ref{p:Corollary on modified operator} in Examples \ref{ex:Complex Jacobian}--\ref{ex:Operators in terms of Hilbert transforms}.

\section{Operators satisfying Assumptions \ref{Assumption on the bilinear operator}--\ref{Assumption on the bilinear operator 2}}
In this section we discuss two classes of operators which satisfy assumption \ref{Assumption on the bilinear operator}: commutators of Calder\'{o}n-Zygmund operators with $\operatorname{BMO}$ functions and paracommutators. Specific examples are then given in \textsection \ref{Examples}.

\subsection{Commutators of Calder\'{o}n--Zygmund operators and $\operatorname{BMO}$ functions} \label{ss:Commutators of Calderon-Zygmund operators}
Recall that
\[\textup{BMO}(\R^n) \defeq \left\{ b \in L^1_{loc}(\R^n) \colon \norm{b}_{\operatorname{BMO}} \defeq \sup_{Q} \Aint_Q \abs{b(x) - b_Q} \, dx < \infty \right\},\]
where the supremum is taken over cubes $Q \subset \R^n$ and $b_Q \defeq \aint_Q b(y) \, dy$. The closure of $C_c^\infty(\R^n)$ in $\textup{BMO}(\R^n)$ is called $\operatorname{CMO}(\R^n)$. By classical results by Fefferman and by Coifman and Weiss, respectively, we have the dualities $[\mathcal{H}^1(\R^n)]^* = \textup{BMO}(\R^n)$ and $[\operatorname{CMO}(\R^n)]^* = \mathscr{H}^1(\R^n)$.

Coifman \& al. showed in ~\cite{CRW76} that when $T$ is a Calder\'{o}n-Zygmund operator with a suitable smooth kernel $\Omega$, we have $\norm{[b,T]}_{L^2 \to L^2} \lesssim \norm{b}_{\operatorname{BMO}}$ for all $b \in \textup{BMO}(\R^n)$. Specifically, $\Omega$ was assumed to be homogeneous of degree zero with vanishing mean over $S^{n-1}$ and to satisfy $\abs{\Omega(x)-\Omega(y)} < \abs{x-y}$ for all $x,y \in S^{n-1}$. They also showed that if $[b,R_j] \colon L^2(\R^n) \to L^2(\R^n)$ is bounded for every $j \in \{1,\ldots,n\}$, then $b \in \operatorname{BMO}(\R^n)$ and $\norm{b}_{\operatorname{BMO}} \lesssim \max_{1 \le j \le n} \norm{[b,R_j]}_{L^2 \to L^2}$. Uchiyama~\cite{Uch78} and Janson~\cite{Jan78} showed independently that in order to obtain $b \in \tp{BMO}(\R^n)$, it suffices to show boundedness of $[b,T]$ in $L^2(\R^n)$ for only one of the kernels $\Omega \not \equiv 0$ in the result of ~\cite{CRW76}. Uchiyama also showed that $[b,T]$ is compact if and only is $b \in \operatorname{CMO}(\R^n)$. Thus $T_b f \defeq [b,T] f$ satisfies Assumption \ref{Assumption on the bilinear operator}.

The two-sided estimate \eqref{Two-sided estimate} has been extended in numerous ways (to multi-parameter and weighted spaces etc.); see e.g.~\cite{GLW20, Hyt21, Wick20} and the references contained therein. In the case of commutators with Calder\'on-Zygmund operators, Hyt\"onen recently proved the estimate $\norm{[b,T]}_{L^2 \to L^2} \simeq \norm{b}_{\operatorname{BMO}}$ under the very weak assumption that the kernel is "non-degenerate". We recall relevant definitions from~\cite{Hyt21}.

\begin{definition} \label{d:omega-kernels}
A measurable function $K \colon \R^n \times \R^n \to \R$ is called an \emph{$\omega$-Calder\'on-Zygmund kernel} if
$K(x,y) \leq c_K \abs{x-y}^{-n}$ whenever $x\neq y$ and
\[\abs{K(x,y)-K(x',y)} + \abs{K(y,x)-K(y,x')} \leq \frac{1}{\abs{x-y}^n} \omega \left( \frac{\abs{x-x'}}{\abs{x-y}} \right)\]
whenever $\abs{x-x'} < \abs{x-y}/2$, where $\omega \colon [0,1) \to [0,\infty)$ is increasing.
\end{definition}

\begin{definition} \label{d:Non-degenerate kernels}
A measurable function $K \colon \R^n \times \R^n \to \R$ is called a \emph{non-degenerate Calder\'on-Zygmund kernel} if $K$ satisfies at least one of the following two conditions:

\begin{enumerate}
\item $K$ is an $\omega$-Calder\'on-Zygmund kernel with $\omega(t) \to 0$ as $t \to 0$ and for every $y \in \R^n$ and $r > 0$, there exists $x \in \complement B(y,r)$ such that $\abs{K(x,y)} \leq cr^{-d}$.

\item $K$ is an homogeneous Calder\'on-Zygmund kernel with $\Omega \in L^1(S^{n-1}) \setminus \{0\}$.
\end{enumerate}
\end{definition}

For compactness of commutators of $\tp{CMO}$ functions and Calder\'{o}n-Zygmund operators on (possibly weighted) $L^2(\R^n)$ we refer to~\cite{CH15,CC13,GZ19,HL21,LL22,Uch78}. In particular, when an homogeneous kernel $\Omega \in L^1(S^{n-1}) \setminus \{0\}$ has vanishing mean and satisfies
\[\sup_{\zeta \in S^{n-1}} \int_{S^{n-1}} \abs{\Omega(\eta)} \left( \tp{ln} \frac{1}{\abs{\eta \cdot \zeta}} \right)^\theta \tp{d} \eta < \infty\]
for some $\theta > 2$, the corresponding commutator satisfies Assumption \ref{Assumption on the bilinear operator}~\cite{CH15,Hyt21}. Assumption \ref{Assumption on the bilinear operator} is verified for commutators with the Cauchy integral operator in ~\cite{LNWW17}.

In many natural instances, an operator $T_b$ satisfying Assumption \ref{Assumption on the bilinear operator} is not precisely a commutator but, for instance, the composition of a commutator with a Calder\'{o}n-Zygmund operator. A natural general framework for such operators is provided by \emph{paracommutators}.

\subsection{Paracommutators} \label{Paracommutators}
We briefly recall basic definitions and results from the theory of paracommutators and refer to ~\cite{JP88}. Paracommutators are operators $T_b(A)$ of the form
\begin{equation} \label{Form of paracommutators}
\widehat{T_b(A) \omega}(\xi) \defeq  (2\pi)^{-n} \int_{\R^n} \hat{b}(\xi-\eta) A(\xi,\eta) \hat{\omega}(\eta) \, d\eta, \quad \xi \in \R^n,
\end{equation}
where $A \colon \R^n \times \R^n \to \R$ and $\omega \colon \R^n \to \R$ are functions and $b \colon \R^n \to \R$ is called the \emph{symbol of $T_b(A)$}.

As the most basic example, $A(\xi,\eta) \equiv 1$ yields (under suitable integrability assumptions) the multiplication operator $T_b \omega = b\omega$. The commutators $T_b(A) = [b,R_j]$ arise via $A(\xi,\eta) = \xi_j/\abs{\xi} - \eta_j/\abs{\eta}$. More generally, if $T$ is a Calder\'on-Zygmund singular integral operator with Fourier symbol $m$ and $A(\xi,\eta) = m(\xi) - m(\eta)$, then $T_b(A) = [b,T]$. Several further examples are presented in~\cite[pp. 469-473]{JP88} and~\cite[pp. 513-519]{PW00}.

In ~\cite{JP88}, Janson and Peetre gave conditions on the function $A$ under which the boundedness of $T_b(A) \colon L^2(\R^n) \to L^2(\R^n)$ is equivalent to the condition $b \in \operatorname{BMO}(\R^n)$. In order to state the result we need some definitions. We do not motivate them here but refer to ~\cite{JP88} for more information.

\begin{definition}
Let $U,V \subset \R^n$. The space $M(U \times V)$ consists of functions $\varphi \in L^\infty(U \times V)$ that admit a representation
\begin{equation} \label{e:Representation}
\varphi(\xi,\eta) = \int_X \alpha(\xi,x) \beta(\eta,x) \, d\mu(x)
\end{equation}
for some $\sigma$-finite measure space $(X,\mu)$ and measurable functions $\alpha \colon U \times X \to \R$ and $\beta \colon V \times X \to \R$ which satisfy
\begin{equation} \label{Admissible representations of Schur multipliers}
\int_X \norm{\alpha(\cdot,x)}_{L^\infty(U)} \norm{\beta(\cdot,x)}_{L^\infty(V)} \, d\mu(x) < \infty.
\end{equation}
Furthermore, $M(U \times V)$ is a Banach algebra in the norm given by minimising the left hand side of \eqref{Admissible representations of Schur multipliers} over all representations \eqref{e:Representation}.
\end{definition}

For every $j \in \mathbb{Z}$ we denote $\Delta_j \defeq \{\xi \in \R^n \colon 2^j \le \abs{\xi} < 2^{j+1}\}$. We list some of the assumptions of ~\cite{JP88} and retain their numbering.

(A0) There exists $r > 1$ such that $A(r\xi,r\eta) = A(\xi,\eta)$ for all $\xi,\eta \in \R^n$.

(A1) $\norm{A}_{M(\Delta_j \times \Delta_k)} \le C$ for all $j,k \in \mathbb{Z}$.

(A2) There exist $A_1,A_2 \in M(\R^n \times \R^n)$ and $\delta > 0$ such that
\begin{align*}
& A(\xi,\eta) = A_1(\xi,\eta) \quad \text{when } \abs{\eta} < \delta \abs{\xi}, \\
& A(\xi,\eta) = A_2(\xi,\eta) \quad \text{when } \abs{\xi} < \delta \abs{\eta}.
\end{align*}

(A3($\alpha$)): There exist $\alpha, \delta > 0$ such that whenever $r < \delta \abs{\xi_0}$, we have
\[\norm{A}_{M(B(\xi_0,r),B(\xi_0,r))} \le C \left( \frac{r}{\abs{\xi_0}} \right)^\alpha.\]

(A5) For every $\xi_0 \neq 0$ there exist $\delta > 0$ and $\eta_0 \in \R^n$ such that $1/A(\xi,\eta) \in M(U \times V)$, where $U = \{\xi \colon \abs{\xi/\abs{\xi} - \xi_0/\abs{\xi_0}} < \delta \text{ and } \abs{\xi} > \abs{\xi_0}\}$ and $V = B(\eta_0, \delta \abs{\xi_0}))$.

\vspace{0.2cm}
We collect the results of~\cite{JP88} that are most relevant to this paper.
\begin{theorem} \label{Janson-Peetre theorem}
Suppose $A \colon \R^n \times \R^n \to \R$ satisfies (A0), (A1), (A2), (A3($\alpha$)) and A(5), where $\alpha > 0$. Then $T_b(A)$ satisfies Assumption \ref{Assumption on the bilinear operator}.
\end{theorem}

By a result of Peng, under the assumptions (A0), (A1), (A3($\alpha$)) and (A5), compactness of $T_b \colon L^2(\R^n) \to L^2(\R^n)$ conversely implies $b \in \tp{CMO}(\R^n)$~\cite{Pen88}. We also mention that under (A0)--(A5), Assumption \ref{Assumption on the bilinear operator 2} is not satisfied in general.
Indeed, denoting $\tilde{b}(x) \defeq b(-x)$ and abusing notation, $\langle T_b(A(\xi,\eta)) \omega, \gamma \rangle = \langle \omega, T_{\tilde{b}}(A(\eta,\xi)) \gamma \rangle$ for all $\omega,\gamma \in L^2(\R^n)$.

If the assumptions of Theorem \ref{Janson-Peetre theorem} and Assumption \ref{Assumption on the bilinear operator 2} hold, we may use Definition \ref{Definition of the quadratic quantity} to define a quadratic operator $Q \colon L^2(\R^n) \to \mathscr{H}^1(\R^n)$. Its form is
\begin{equation} \label{Form of compensated quantities}
Q_{\tilde{A}}(\omega,\gamma)(x) = \frac{1}{(2\pi)^{2n}} \int_{\R^n} \int_{\R^n} \hat{\omega}(\xi) \hat{\gamma}(\eta) \tilde{A}(\xi,\eta) e^{-i (\xi+\eta) \cdot x} \, d\xi \, d\eta,
\end{equation}
where $\tilde{A}(\eta,-\xi) = A(\xi,\eta)$.
Conversely, $Q_{\tilde{A}}$, defined via \eqref{Form of compensated quantities}, gives rise to a paracommutator $T_b(A)$ defined by \eqref{Form of paracommutators}. These and many other relations between $T_b(A)$ and $Q_{\tilde{A}}$ are explored in~\cite{PW00}.

\subsection{Specific operators} \label{Examples}
We list some concrete quadratic operators which arise via Assumption \ref{Assumption on the bilinear operator}; numerous further examples are presented in~\cite{JP88, PW00}.

\begin{example} \label{ex:Complex Jacobian}
Let $u \in \dot{W}^{1,2}(\C,\C)$. When $u_z \defeq 2^{-1} (\partial_x-i \partial_y)(u_1+iu_2)$ and $u_{\bar{z}} \defeq 2^{-1} (\partial_x+i \partial_y)(u_1+iu_2)$ are the Wirtinger derivatives of $u$ and $\mathcal{S} \colon L^2(\C,\C) \to L^2(\C,\C)$ is the Beurling transform, we have $\mathcal{S} u_{\bar{z}} = u_z$. Therefore, the Jacobian $J u = |u_z|^2 - |u_{\bar{z}}|^2$ corresponds to $Q \omega \defeq |\S \omega|^2 - |\omega|^2$ via the isometric isomorphism $u \mapsto u_{\bar{z}} \eqdef \omega \colon \dot{W}^{1,2}(\C,\C) \to \L2$ whose inverse is the Cauchy transform $\mathcal{C}$. (We identify elements of $\dot{W}^{1,2}(\C,\C)$ that differ by a constant and set $\norm{u}_{\dot{W}^{1,2}} \defeq \norm{u_{\bar{z}}}_{L^2} = \norm{D u}_{L^2}/2$.) The quadratic operator $Q$ arises via the formula $\langle b, Q \omega \rangle_{\operatorname{BMO}-\mathcal{H}^1} = \langle T_b \omega, \omega \rangle_{L^2}$, where $T_b \omega \defeq \overline{(S b - bS) \overline{S \omega}}$ satisfies Assumption \ref{Assumption on the bilinear operator} for $H = L^2(\C,\C)$ and $X = \tp{CMO}(\C)$ (see ~\cite{Lin15}).

It is instructive to also consider the planar Jacobian in real variable notation. Define
\begin{equation} \label{e:Tb}
T_b \defeq R_1 [b, R_2] - R_2 [b, R_1] \colon L^2(\R^2) \to L^2(\R^2).
\end{equation}
Now $T_b^* = - T_b \neq 0$ and thus $T_b$ is not self-adjoint. However, the operator $\widetilde{T_b} \colon L^2(\R^2,\R^2) \to L^2(\R^2,\R^2)$, defined in Proposition \ref{p:Corollary on modified operator}, generates the Jacobian:
\begin{align*}
\int_{\R^2} \widetilde{T_b}(\omega_1,\omega_2) \cdot (\omega_1,\omega_2)
&= 2 \langle b, R_1 \omega_1 R_2 \omega_2 - R_2 \omega_1 R_1 \omega_2 \rangle_{\operatorname{BMO}-\mathcal{H}^1} \\
&= 2 \langle b, \partial_1 v_1 \partial_2 v_2 - \partial_2 v_1 \partial_1 v_2 \rangle_{\text{BMO}-\mathcal{H}^1},
\end{align*}
where $v = \Lambda \omega \defeq (-\Delta)^{-1/2} \omega \in \dot{W}^{1,2}(\R^2)$. This example serves to further motivate Assumption \ref{Assumption on the bilinear operator 2}.
\end{example}

\begin{example} \label{ex:Operators in terms of Hilbert transforms}
On the real line, two quadratic $\mathcal{H}^1$-integrable quantities arise as follows. Isomorphically, $\mathcal{H}^p(\R) \cong \mathcal{H}^p(\R,\C) \defeq \{g + i Hg \colon g, Hg \in L^p(\R)\}$
for all $p \in (0,\infty)$, where $H$ is the Hilbert transform. Given $\omega \in L^2(\R)$ we set
\[(\omega + i H\omega)^2 = \omega^2 - (H\omega)^2 + 2 i \omega H \omega \eqdef Q_1 \omega + i Q_2 \omega.\]
As is well-known, $Q_1,Q_2 \colon L^2(\R) \to \mathcal{H}^1(\R)$ are not surjective but their G\^{a}teaux derivatives $(Q_1)_\omega' \gamma = \omega \gamma - H\omega H\gamma$ and $(Q_2)'_\omega \gamma = \omega H \gamma + \gamma H \omega$ are--in other words,
\begin{equation} \label{Product of Hardy spaces}
\{\alpha^2 \colon \alpha \in \mathcal{H}^2(\R,\C)\} \subsetneq \{\alpha \beta \colon \alpha, \beta \in \mathcal{H}^2(\R,\C)\} = \mathcal{H}^1(\R,\C).
\end{equation}
We give a proof of \eqref{Product of Hardy spaces} in Appendix \ref{s:Appendix} for the reader's convenience.

The operator $Q_1$ arises via the bounded linear operator $T_b \colon L^2(\R) \to L^2(\R)$, $T_b \defeq [H, b] H$ by the formula $\langle b, Q_1 \omega \rangle_{\operatorname{BMO}-\mathcal{H}^1} \defeq \int_{\R} \omega T_b \omega$. Now $T_b$ satisfies Assumptions \ref{Assumption on the bilinear operator}--\ref{Assumption on the bilinear operator 2} but $Q_1 \colon L^2(\R) \to \mathcal{H}^1(\R)$ is not surjective. This illustrates the fact that the modified operator $\tilde{T}_b$ (see Proposition \ref{p:Corollary on modified operator}) might be needed to ensure surjectivity even if $T_b$ satisfies Assumptions \ref{Assumption on the bilinear operator}--\ref{Assumption on the bilinear operator 2}.
\end{example}

\begin{example}
In ~\cite{Wu96}, Wu considered higher-dimensional variants of the operators $Q_j$ of Example \ref{ex:Operators in terms of Hilbert transforms} in Clifford algebras. When $\Omega = \omega + \sum_{j=1}^n R_j \omega \mathbf{e}_j \in L^2(\R^n,\C_{(n)})$ and $\Gamma = \gamma + \sum_{j=1}^n R_j \gamma \mathbf{e}_j \in L^2(\R^n,\C_{(n)})$, their product can be written as
\begin{equation} \label{e:Factorisation in upper half space}
\Omega \Gamma = \omega \gamma - \sum_{j=1}^n R_j \omega R_j \gamma + \sum_{j=1}^n (\omega R_j \gamma + \gamma R_j \omega) \mathbf{e}_j + \sum_{j < k} (R_j \omega R_k \gamma - R_k \omega R_j \gamma) \mathbf{e}_{jk}.
\end{equation}
Each component of $\Omega \Gamma$ then belongs to $\mathcal{H}^1(\R^n)$~\cite[Theorem 12.3.1]{Wu96}. Note that when $n=2$, the last term of \eqref{e:Factorisation in upper half space} is a Jacobian in disguise: $R_1 \omega R_2 \gamma - R_2 \omega R_1 \gamma = Jv$ for $v = (\Lambda^{-1} \omega,\Lambda^{-1} \gamma) \in \dot{W}^{1,2}(\R^2,\R^2)$. In analogy to \eqref{Product of Hardy spaces}, Wu conjectured that $\mathcal{H}^1(\R^n) = \{\omega \gamma - \sum_{j=1}^n R_j \omega R_j \gamma \colon \omega,\gamma \in L^2(\R^n)\}$~\cite[p. 237]{Wu96}.

Wu also studied more general combinations of Riesz transforms that arise via $A(\xi,\eta) = 1 - (\xi \cdot \eta)^m/(\abs{\xi} \abs{\eta})^m$, $m \in \N$, and showed that each such $A$ satisfies the assumptions of Theorem \ref{Janson-Peetre theorem}. As an example, $T_b \defeq [R, b] \cdot R \colon L^2(\R^n) \to L^2(\R^n)$ generates
\[Q f = \abs{\omega}^2 - \abs{R \omega}^2 = \abs{\Lambda u}^2 - \abs{\nabla u}^2,\]
where $\Lambda u = f$. In this case, $A(\xi,\eta) = 1 - \eta \cdot \xi/(\abs{\eta} \abs{\xi})$. Thus Wu's conjecture says that the G\^{a}teaux derivative $Q'_\omega \gamma = 2(\omega \gamma - R\omega \cdot R\gamma) = 2(\Lambda u \Lambda v - \nabla u \cdot \nabla v)$ gives a surjective bilinear operator $L^2(\R^n) \times L^2(\R^n) \to \mathcal{H}^1(\R^n)$.

The planar Monge-Amp\`ere equation arises as follows. Define $T_b \colon L^2(\R^2) \to L^2(\R^2)$ by
\[T_b \omega \defeq \frac{[R_{11}, b] R_{22} \omega + [R_{22}, b] R_{11} \omega - 2 [R_{12}, b] R_{12} \omega}{2}.\]
Then $Q \colon L^2(\R^2) \to \mathcal{H}^1(\R^2)$ is of the form $Q \omega = R_{11} \omega R_{22} \omega - R_{12} \omega R_{21} \omega$. Using the isomorphism $-\Delta \colon \dot{W}^{2,2}(\R^2) \to L^2(\R^2)$ we may write the Hessian as $\mathcal{H} = Q \circ (-\Delta)^{-1} \colon \dot{W}^{2,2}(\R^2) \to \mathcal{H}^1(\R^2)$. In this case, $A(\xi,\eta) = 1 - (\xi \cdot \eta)^2/(\abs{\xi}^2 \abs{\eta}^2)$.
\end{example}

\section{Banach space geometric considerations} \label{Banach space geometric preliminaries}
This subsection is devoted to proving various claims presented in \textsection \ref{ss:Overall strategy and aim}. We assume that $T$ satisfies Assumptions \ref{Assumption on the bilinear operator}--\ref{Assumption on the bilinear operator 2}.

\subsection{Minimum norm solutions}
We begin by noting that if the equation $Q\omega = f$ has a solution, then the direct method gives a minimum norm solution.

\begin{proposition} \label{p:Existence of minimum norm solutions}
Suppose $f \in X^*_Q$, $\gamma \in H$ and $Q \gamma = f$. Then there exists $\omega \in H$ with $Q \omega = f$ and $\norm{\omega}_H = \min_{Q \gamma = f} \norm{\gamma}_H$.
\end{proposition}

\begin{proof}
Choose a minimising sequence so that $Q \omega_j = f$ for every $j \in \N$ and $\lim_{j \to \infty} \norm{\omega_j}_H = \inf_{Q \gamma = f} \norm{\gamma}_H$. Since $\mathbb{B}_H$ is sequentially weakly compact and $Q$ is weak-to-weak$^*$ continuous, we have $\omega_j \rightharpoonup \omega$ and $Q \omega_j \overset{*}{\rightharpoonup} Q \omega$ for a subsequence, and $\omega$ is the sought minimum norm solution.
\end{proof}

We then show that surjectivity of $Q \colon H \to X^*_Q$ would follow from the weak$^*$ density of the range $Q(H)$ in $X^*_Q$ combined with a suitable a priori estimate.

\begin{proposition} \label{p:Weak continuity proposition}
Suppose that $\overline{Q(H)}^{\operatorname{w}-*} = X^*_Q$ and that $\norm{\omega}_H^2 \lesssim \norm{Q \omega}_{X^*_Q}$ for every minimum norm solution. Then $Q(H) = X^*_Q$.
\end{proposition}

\begin{proof}
Let $f \in X^*$ and choose minimum norm solutions $\omega_j \in H$ with $Q \omega_j \overset{*}{\rightharpoonup} \omega$. For large enough $j \in \N$ we have $\norm{\omega_j}_H^2 \lesssim \norm{f}_{X^*} + 1$. After passing to a subsequence, a weak limit $\omega$ of $(\omega_j)_{j=1}^\infty$ satisfies $Q \omega = f$.
\end{proof}

We also mention a version of the Banach-Schauder open mapping theorem for multilinear (and more general) operators from~\cite{GKL20}. We say that $Q$ enjoys \emph{generalised translation invariance} if there exist isometric isomorphisms $\sigma_j \in H^H$ and $\sigma_j \in (X_Q^*)^{X_Q^*}$, $j \in \N$, such that $Q \circ \sigma_j = \sigma_j \circ Q$ for all $j \in \N$ and $\sigma_j f \overset{*}{\rightharpoonup} 0$ as $j \to \infty$ for all $f \in X_Q^*$. The example one should have in mind is $\sigma_j \omega(x) = \omega(x-je)$ and $\sigma_j f(x) \defeq f(x-je)$ for a non-zero vector $e$.

\begin{theorem} \label{t:GKL open mapping theorem}
Suppose $Q$ enjoys generalised translation invariance. Then one of the following claims holds:
\begin{enumerate}
\item $Q(H) = X_Q^*$ and every minimum norm solution satisfies $\norm{\omega}_H^2 \lesssim \norm{Q\omega}_{X_Q^*}$.

\item $Q(H)$ is meagre in $X_Q^*$.
\end{enumerate}
\end{theorem}

Theorem \ref{t:GKL open mapping theorem} is primarily a practical tool for showing \emph{non}-surjectivity of $Q \colon H \to X^*_Q$; disproving the \emph{a priori} estimate $\norm{\omega}_H^2 \lesssim \norm{Q\omega}_{X_Q^*}$ is a markedly less daunting task than disproving surjectivity~\cite{GKL20}.

\subsection{Definition and existence of Lagrange multipliers}
When $\omega \in H$ is a minimum norm solution, it is natural to look for a Lagrange multiplier of $\omega$, as discussed in \textsection \ref{ss:Overall strategy and aim}. The Lagrange multiplier condition \eqref{Lagrange multiplier condition} can be written more concisely as follows: $b \in X^{**}_Q$ is a Lagrange multiplier of $\omega \in H$ if
\begin{equation} \label{Lagrange multiplier condition 2}
2 \mathbf{Re} \langle \omega, \varphi \rangle_H = \langle b, Q'_\omega \varphi \rangle_{X_Q^{**}-X_Q^*} \quad \text{ for every } \varphi \in H.
\end{equation}
Yet more consicely, $(Q_\omega')^* b = 2 \omega$. The existence of $b$ then means that $\omega \in \operatorname{ran}((Q_f')^*) = \ker(Q_f')^\perp$

The standard tool for showing the existence of a Lagrange multiplier in a Banach space is the Liusternik-Schnirelman theorem which, in this setting, says that if $Q_\omega'$ maps $H$ onto $X_Q^*$, then $\omega$ possesses a Lagrange multiplier. However, in the cases of interest to us, the Liusternik-Schnirelman theorem is not available, as shown belov.

\begin{proposition} \label{p:Liusternik-Schnirelman proposition}
Suppose $X^*_Q$ is not isomorphic to a Hilbert space. Then we have $\{Q'_\omega \gamma \colon \gamma \in H\} \subsetneq X^*_Q$ for every $\omega \in H$.
\end{proposition}

\begin{proof}
Seeking a contradiction, suppose $Q'_\omega H = X^*$. Thus $Q'_\omega \colon \operatorname{ker}(Q')^\perp \to X^*$ is an isomorphism from a Hilbert space onto $X^*$.
\end{proof}

\begin{proposition} \label{p:Proposition on existence of Lagrange multipliers}
Every $b \in \mathbb{S}_{X_Q}$ is a Lagrange multiplier of some $\omega \in \mathcal{A}$.
\end{proposition}

\begin{proof}
Choose a maximizing sequence: $\langle T_b \omega_j, \omega_j \rangle_H \to 1$. For a subsequence, $\omega_j \rightharpoonup \omega$ in $H$, and so $T_b \omega_j \to T_b \omega$, giving $1 = \langle T_b \omega, \omega \rangle_H = \langle b, Q \omega \rangle_{X_Q-X_Q^*} \le \norm{Q f}_{X_Q^*} \le \norm{f}_H^2 \le 1$, so that $\omega \in \mathscr{A}$.
\end{proof}

\begin{definition} \label{d:Definition of a James boundary}
When $Y$ is a real Banach space, a set $A \subset \mathbb{S}_{Y^*}$ is call a \emph{James boundary of $Y$} if for every $y \in \mathbb{S}_Y$ there exists $y^* \in A$ such that $\langle y^*, y \rangle_{Y^*-Y} = 1$.
\end{definition}

Proposition \ref{p:Proposition on existence of Lagrange multipliers} then says that $Q(\mathscr{A})$ is a James boundary of $X_Q$. By Godefroy's Theorem (see ~\cite[Theorem 3.46]{FHHMPZ}), if $A \subset \mathbb{S}_{Y^*}$ is a separable James boundary of $Y$, then $\overline{\operatorname{co}}(A) = \mathbb{B}_{Y^*}$.

\begin{corollary} \label{c:Corollary on James boundaries}
$\overline{\operatorname{co}}(Q(\mathscr{A})) = \mathbb{B}_{X^*_Q}$.
\end{corollary}

\begin{corollary} \label{c:Norm corollary}
$\norm{f}_{X_Q^*} = \inf\{\sum_{j=1}^\infty \norm{\omega_j}_H^2 \colon f = \sum_{j=1}^\infty Q \omega_j\}$ for all $f \in X_Q^*$.
\end{corollary}

\subsection{Further properties of Lagrange multipliers}
The Lagrange multiplier condition \eqref{Lagrange multiplier condition 2} has several useful equivalent characterisations, and we collect two of them in the following proposition.
\begin{proposition} \label{p:Proposition on equivalent characterizations of the Lagrange multiplier condition}
Whenever $b \in \mathbb{S}_{X_Q^{**}}$ and $\omega \in \mathbb{S}_H$, the following conditions are equivalent:
\renewcommand{\labelenumi}{(\roman{enumi})}
\begin{enumerate}
\item $b$ is a Lagrange multiplier of $\omega$, i.e., \eqref{Lagrange multiplier condition 2} holds.

\item $\langle b, Q \omega \rangle_{X_Q^{**}-X_Q^*} = 1$.

\item $\omega \in \ker(I-T_b)$.
\end{enumerate}
\noindent If (i)--(iii) hold, then $\omega$ is a minimum norm solution and belongs to $\mathscr{A}$.
\end{proposition}

\begin{proof}
We first prove (i) $\Leftrightarrow$ (ii). Suppose (ii) holds and fix $\varphi \in H$ and small $\delta > 0$. The function $I \colon [-\delta,\delta] \to \R$, $I(\epsilon) \defeq \langle b, Q(\omega + \epsilon \varphi) \rangle_{X_Q^{**}-X_Q^*}/\norm{\omega+\epsilon \varphi}_H^2$ is maximised at $\epsilon = 0$, and therefore $I'(0) = \langle b, Q'_\omega \varphi \rangle_{X_Q^{**}-X_k^*} - 2 \mathbf{Re} \langle \omega, \varphi \rangle_H = 0$, giving (i). The direction (i) $\Rightarrow$ (ii) is proved by setting $\varphi = \omega$. We then prove (ii) $\Leftrightarrow$ (iii). First, (ii) gives $\langle \omega, T_b \omega \rangle_H = 1$. Since $\norm{T_b}_{H \to H} = 1$ and $H$ is strictly convex, we conclude that (iii) holds. On the other hand, if (iii) holds, then $\langle b, Q \omega \rangle_{X_Q^{**}-X_Q^*} = \langle \omega, T_b \omega \rangle_H = \langle \omega, \omega \rangle_H = 1$.
\end{proof}

In many cases, conditions (i)--(iii) can be supplemented by Euler-Lagrange equations for a suitable potential. In the case of the planar Jacobian, denoting $u_{\bar{z}} = \omega$ as before, (i)--(iii) are equivalent to $u_{z\bar{z}} = (bu_{\bar{z}})_z - (bu_z)_{\bar{z}}$~\cite[Proposition 4.8]{Lin15}.

\begin{lemma} \label{l:Isometric isomorphism}
$(X^{**}, \norm{T_\cdot}_{H \to H}) = (X_Q)^{**}$ isometrically.
\end{lemma}

\begin{proof}
Let $b \in X^{**}$. Since $Q(\mathbb{B}_H) \subset \mathbb{B}_{X_Q^*} = \overline{\operatorname{co}}(Q \mathscr{A})$, we get
\[\norm{T_b}_{H \to H} = \sup_{\norm{f}_H = 1} \langle b, Q f \rangle_{X^{**}-X^*}
= \sup_{h \in \operatorname{co}(Q \mathscr{A})} \langle b, h \rangle_{X^{**}-X^*} = \norm{b}_{(X^*,\norm{\cdot}_{X_Q^*})^*}.\]
\end{proof}

\begin{proposition} \label{p:Lagrange multiplier equivalent condition}
Suppose $b \in \mathbb{S}_{X_Q^{**}}$. The following conditions are equivalent.
\renewcommand{\labelenumi}{(\roman{enumi})}
\begin{enumerate}
\item $b$ is a Lagrange multiplier of some $\omega \in \mathscr{A}$.

\item $b$ is norm-attaining.
\end{enumerate}
\end{proposition}

\begin{proof}
If $b$ is a Lagrange multiplier of $\omega \in \mathscr{A}$, then $T_b \omega = \omega$ and therefore $\langle b, Q \omega \rangle_{X_Q^{**}-X_Q^*} = \langle T_b \omega, \omega \rangle_H = 1$.

Conversely, if $b$ is norm-attaining, denote $D^{-1}(b) \defeq \{f \in \mathbb{S}_{X_Q^*} \colon \langle b, f \rangle_{X_Q^{**}-X_Q^*} = 1\}$. Since $D^{-1}(b)$ is closed, bounded and convex, Theorem \ref{Bessaga-Pelczynski theorem} implies that $D^{-1}(b)$ contains an extreme point $f$. The definition of $D^{-1}(b)$ implies that $f$ is also an extreme point of $\mathbb{S}_{X_Q^*}$, and Lemma \ref{l:Krein-Milman result} then gives $f \in Q(\mathscr{A})$.
\end{proof}

We also characterise elements of $\mathbb{S}_H$ that possess a Lagrange multiplier.

\begin{proposition} \label{p:Having a Lagrange multiplier}
Let $\omega \in \mathbb{S}_H$. The following conditions are equivalent.
\renewcommand{\labelenumi}{(\roman{enumi})}
\begin{enumerate}
\item $\omega$ has a Lagrange multiplier $b \in \mathbb{S}_{X_Q^{**}}$.

\item $\omega \in \mathscr{A}$.
\end{enumerate}
\end{proposition}

\begin{proof}
If $\omega \in \mathbb{S}_H$ has a Lagrange multiplier $b \in \mathbb{S}_{X_Q^{**}}$, then $\langle b, Q \omega \rangle_{X_Q^{**}-X_Q^*} = \langle T_b \omega, \omega \rangle_H = \norm{\omega}_H^2 = 1$ so that $\omega \in \mathscr{A}$.

Conversely, let $\omega \in \mathscr{A}$. Since $Q \omega \in \mathbb{S}_{X_Q^*}$, there exists $b \in \mathbb{S}_{X_Q^{**}}$ such that $\langle b, Q \omega \rangle_{X_Q^{**}-X_Q^*} = 1$. Now $b$ is a Lagrange multiplier of $\omega$.
\end{proof}

Using the results above we collect many equivalent formulations of Question \ref{q:Lindberg}.

\begin{corollary} \label{c:Collection of conditions}
If $Q(H)$ is dense in $X_Q^*$, the following conditions are equivalent:
\renewcommand{\labelenumi}{(\roman{enumi})}
\begin{enumerate}
\item $Q(\mathscr{A}) = \mathbb{S}_{X_Q^*}$.

\item $\E = \norm{\cdot}_{X_Q^*}$.

\item Every minimum norm solution satisfies $\norm{\omega}_H^2 = \norm{Q \omega}_{X_Q^*}$.

\item Every minimum norm solution has a Lagrange multiplier in $\mathbb{S}_{X_Q^{**}}$.
\end{enumerate}
\end{corollary}



\section{The proof of Theorem \ref{t:Theorem on uniqueness of minimum norm solutions}} \label{s:The proof of Theorem on uniqueness of minimum norm solutions}
We get the proofs of Theorems \ref{t:Theorem on duality mapping} and \ref{t:Theorem on uniqueness of minimum norm solutions} underway by proving \eqref{Krein-Milman set inclusion}.

\begin{lemma} \label{l:Krein-Milman result}
If Assumptions \ref{Assumption on the bilinear operator}--\ref{Assumption on the bilinear operator 2} holds, then $Q(\mathscr{A}) \supset \operatorname{ext}(\mathbb{B}_{X^*_Q})$. Furthermore, $Q(\mathscr{A})$ is closed in the relative weak$^*$ topology of $\mathbb{S}_{X^*_Q}$. In particular, it is norm closed.
\end{lemma}

\begin{proof}
We first show that $Q(\mathscr{A})$ is relatively weak$^*$ (sequentially) closed in $\mathbb{S}_{X_Q^*}$. Since $X_Q$ is separable, it suffices to consider sequences instead of nets. Suppose $\omega_j \in \mathscr{A}$ for every $j \in \N$ and $Q \omega_j \overset{*}{\rightharpoonup} f \in \mathbb{S}_{X_Q^*}$. We claim that $f = Q \omega$ for some $\omega \in \mathscr{A}$.

Passing to a subsequence, $\omega_j \rightharpoonup \omega \in \mathbb{B}_{H}$, since $\mathbb{B}_{H}$ is weakly compact. Thus $Q \omega_j \overset{*}{\rightharpoonup} Q \omega$, and so it suffices to show that $\omega \in \mathscr{A}$. The chain of inequalities $\norm{\omega}_H^2 \le \liminf_{j \to \infty} \norm{\omega_j}_H^2 = 1 = \norm{Q \omega}_{X^*_Q} \le \norm{\omega}_H^2$ implies that $\omega \in \mathscr{A}$.

We then show that every extreme point of $\mathbb{B}_{X_Q^*}$ lies in $Q(\mathscr{A})$. In Theorem \ref{Milman's theorem}, choose $K$ to be the weak$^*$ closure of $Q(\mathscr{A})$ in $\mathbb{B}_{X_Q^*}$. By Corollary \ref{c:Corollary on James boundaries}, $\overline{\operatorname{co}}^{\operatorname{w}^*}(K) = \mathbb{B}_{X_Q^*}$. Theorem \ref{Milman's theorem} implies that $\operatorname{ext}(\mathbb{B}_{X_Q^*}) \subset K$. Let now $f \in \operatorname{ext}(B_{X_Q^*})$ and choose $\omega_j \in \mathscr{A}$ such that $Q \omega_j \overset{*}{\rightharpoonup} f$. Since $Q(\mathscr{A})$ is relatively weak$^*$ closed in $\mathbb{S}_{X_Q^*}$, we obtain $f \in Q(\mathscr{A})$.
\end{proof}

We recall the statement of Theorem \ref{t:Theorem on uniqueness of minimum norm solutions}.

\begin{claim} \label{t:Theorem on uniqueness of minimum norm solutions2}
Suppose Assumptions \ref{Assumption on the bilinear operator}--\ref{Assumption on the bilinear operator 2} and \ref{Assumption 2} hold, and let $f \in \operatorname{ext}(\mathbb{B }_{X_Q^*})$. Then the minimum norm solution $\omega \in \A$ of $Q \omega = f$ is unique up to multiplication by $c \in \mathbb{S}_{\mathbb{K}}$.
\end{claim}

\begin{proof}
Seeking a contradiction, suppose $f \in \operatorname{ext}(\mathbb{B}_{X_Q^*})$ and $Q \omega = Q \gamma = f$, where $\omega,\gamma \in \mathscr{A}$ are not equal up to a rotation. Thus we may write $\gamma = s \omega + t \omega^\perp$, where $s,t \in \mathbb{K}$ with $\abs{s}^2 + \abs{t}^2 = 1$, $t \neq 0$ and $\omega^\perp \in \mathbb{S}_H \cap \{\omega\}^\perp$. We intend to show that
\begin{equation} \label{e:Aimed equality}
Q (\omega^\perp) = Q \omega.
\end{equation}
Once \eqref{e:Aimed equality} is proved, Assumption \ref{Assumption 2} yields $c \in \mathbb{S}_{\mathbb{K}}$ such that $Q'_\omega (c \omega^\perp) \neq 0$. Denoting $\psi^\pm \defeq (\omega\pm c \omega^\perp)/\sqrt{2} \in \SSS_H$ we therefore get $Q \omega = (Q \psi^+ + Q \psi^-)/2$ but $\B_{X_Q^*} \ni Q \psi^\pm = Q \omega \pm Q'_\omega (c \omega^\perp)/2 \neq Q \omega$, thereby obtaining the sought contradiction.

Formula \eqref{e:Aimed equality} clearly holds if $s = 0$, so assume $s \neq 0$. We denote $\tilde{\omega}^\perp \defeq (\abs{st}/s\bar{t}) \omega^\perp$ and $\tilde{\gamma} \defeq (\bar{s}/\abs{s}) \gamma = \abs{s} \omega + \abs{t} \tilde{\omega}^\perp \in \A$. Now $Q \omega = Q \tilde{\gamma} = \abs{s}^2 Q \omega + \abs{t}^2 Q \tilde{\omega}^\perp + \abs{s} \abs{t} Q'_\omega \tilde{\omega}^\perp$, which yields
\begin{equation} \label{e:Jacobian of perpendicular map}
Q \tilde{\omega}^\perp = Q \omega - \frac{\abs{s}}{ \abs{t}} Q'_\omega \tilde{\omega}^\perp.
\end{equation}
Since $Q \tilde{\omega}^\perp = Q \omega^\perp$, it therefore suffices to show that $Q'_\omega \tilde{\omega}^\perp = 0$.

Whenever $\delta, \epsilon \in \R$ with $\delta^2 + \epsilon^2 = 1$, we have
\begin{equation} \label{e:Q of sum}
Q (\delta \omega + \epsilon \tilde{\omega}^\perp) = \delta^2 Q_\omega + \delta \epsilon Q'_\omega \tilde{\omega}^\perp + \epsilon^2 Q \tilde{\omega}^\perp  = Q\omega +(\delta \epsilon - \epsilon^2 \abs{s} \abs{t}^{-1}) Q'_\omega \tilde{\omega}^\perp.\end{equation}
By choosing any $r > 0$ and letting $\delta_\pm = \epsilon(\abs{s}/\abs{t}\pm r)$ with $\delta_\pm^2+\epsilon^2 = 1$ we get
\begin{align}
& \delta_+ \epsilon-\epsilon^2\abs{s}/\abs{t} = - (\delta_- \epsilon-\epsilon^2\abs{s}/\abs{t}) \neq 0, \label{e:Coefficients} \\
& Q(\delta_+ \omega + \epsilon_1 \tilde{\omega}^\perp) + Q(\delta_- \omega + \epsilon_2 \tilde{\omega}^\perp) = 2 Q \omega. \label{e:Sum of two terms}
\end{align}
Since $Q \omega \in \ext(\B_{X_Q^*})$, we conclude from \eqref{e:Q of sum}--\eqref{e:Sum of two terms} that $Q'_\omega \tilde{\omega}^\perp = 0$, and now \eqref{e:Jacobian of perpendicular map} and the definition of $\tilde{\omega}^\perp$ imply \eqref{e:Aimed equality}.
\end{proof}

\begin{corollary} \label{c:Corollary on points of Frechet differentiability}
Suppose Assumptions \ref{Assumption on the bilinear operator}--\ref{Assumption on the bilinear operator 2} and \ref{Assumption 2} hold and $b \in \mathbb{S}_{X_Q}$ is a point of Fr\'{e}chet differentiability. Then $\ker(I-T_b)$ is one-dimensional.
\end{corollary}

\begin{proposition} \label{p:Assumption satisfied}
The operator $\omega \mapsto Q \omega = \abs{\S\omega}^2-\abs{\omega}^2 \colon L^2(\C,\C) \to \mathcal{H}^1(\C)$ satisfies Assumption \ref{Assumption 2}, whereas $(\omega,\gamma) \mapsto \tilde{Q}(\omega,\gamma) \defeq G'_\omega \gamma \colon L^2(\C,\C)^2 \to \H^1(\C)$ does not.
\end{proposition}

\begin{proof}
We prove the first statement by contradiction. Suppose $\omega,\gamma \in \mathscr{A}$ satisfy $Q \omega = Q \gamma$ and $Q_\omega'(c \gamma) = 2 \mathbf{Re} (S \omega \overline{S (c \gamma)} - \omega \overline{c \gamma}) = 0$ for every $c \in \mathbb{S}_\mathbb{K}$. Thus $S \omega \overline{S \gamma} - \omega \overline{\gamma} = 0$. Fix $\delta > 0$ and denote $X_\delta \defeq \{z \in \mathbb{C} \colon \abs{\omega(z)} \in [\delta,1/\delta]\}$. In $X_\delta$ we use the formula $S \omega \overline{S \gamma} - \omega \overline{\gamma} = 0$ to write
\[\abs{\mathcal{S} \omega}^2 - \abs{\omega}^2 = \abs{\mathcal{S} \gamma}^2 - \abs{\gamma}^2 = - \frac{\abs{\mathcal{S} \gamma}^2}{\abs{\omega}^2} (\abs{\mathcal{S} \omega}^2 - \abs{\omega}^2),\]
so that $\abs{\mathcal{S} \omega}^2 - \abs{\omega}^2 = 0$ in $X_\delta$. Thus
\[\int_{\omega \neq 0} \abs{\abs{\mathcal{S} \omega}^2 - \abs{\omega}^2} = \lim_{\delta \to 0} \int_{X_\delta} \abs{\abs{\mathcal{S} \omega}^2 - \abs{\omega}^2} = 0,\]
and therefore $\int_{\omega = 0} \abs{\mathcal{S} \omega}^2 = 1 - \int_{\omega \neq 0} \abs{\mathcal{S} \omega}^2 = 1 - \int_{\omega \neq 0} \abs{\omega}^2 = 0$. We conclude that $\abs{\mathcal{S} \omega}^2 - \abs{\omega}^2 = 0$, which yields the desired contradiction.

We show a strengthened version of the second statement: whenever $(\omega,\gamma) \in \A$, there exists $(\varphi,\psi) \in \A$ such that $\tilde{Q}(\omega,\gamma) = \tilde{Q}(\varphi,\psi)$ but $\tilde{Q}'_{(\omega,\gamma)} [c(\varphi,\psi)] = 0$ for all $c \in S^1$. Given $(\omega,\gamma)$, we choose $(\varphi,\psi) = (-\overline{\S \gamma},\overline{\S\omega})$ and use the facts that $\S \colon L^2(\C,\C) \to L^2(\C,\C)$ is an isometry and $\S \overline{\S \eta} = \bar{\eta}$ for all $\eta \in L^2(\C,\C)$. 
\end{proof}

An almost identical proof applies to Example \ref{ex:Operators in terms of Hilbert transforms}, but instead of $(\varphi,\psi) = (-\overline{\S \gamma},\overline{\S \omega})$ we set $(\varphi,\psi) = (-H\omega,H\gamma)$.

\begin{proposition} \label{p:Assumption satisfied in Hilbert case}
The operator $\omega \mapsto \omega^2-(H\omega)^2 \colon L^2(\R) \to \mathcal{H}^1(\R)$ satisfies Assumption \ref{Assumption 2}, whereas $(\omega,\gamma) \mapsto \omega \gamma - H\omega H\gamma \colon L^2(\R)^2 \to \H^1(\R)$ does not.
\end{proposition}

\begin{remark} \label{r:Symmetries}
The proof of Proposition \ref{p:Assumption satisfied} uses a central difference of the operator $\omega \mapsto Q \omega = \abs{\S\omega}^2-\abs{\omega}^2$ and its G\^ateaux derivative $\tilde{Q}$. Given $\omega \in \A$, the set $\{\varphi \in \A \colon Q \varphi = Q \omega\}$ sometimes consists only of rotations of $\omega$, whereas all the corresponding sets $\{(\varphi,\psi) \in \tilde{\A} \colon \tilde{Q}(\varphi,\psi) = \tilde{Q}(\omega,\gamma)\}$ are invariant under the linear transformation $L(\omega,\gamma) \defeq (-\overline{S \gamma},\overline{\S \omega})$. Note that $L \circ L = -\tp{id}$. Similar remarks apply to the sets of minimum norm solutions for a given datum $f \in \H^1(\C)$.
\end{remark}

\section{The proof of Theorem \ref{t:Theorem on duality mapping}} \label{s:Proof of Theorem on duality mapping}

\subsection{Claims (i)--(iii)}
In this subsection, Assumptions \ref{Assumption on the bilinear operator}--\ref{Assumption on the bilinear operator 2} are in effect. Before recalling Claims (i)--(iii) we note the following consequence of Corollary \ref{c:Collection of conditions}.
\begin{lemma} \label{l:Lemma on Euclidean unit sphere}
Let $b \in \mathbb{S}_{X_Q}$. Then $\{f \in \mathscr{A} \colon Q f \in D(b)\} = \ker(I-T_b) \cap \mathbb{S}_H$.
\end{lemma}

\begin{claim}
For every $b \in \mathbb{S}_{X_Q}$, $D(b) \cap Q(\mathscr{A})$ contains $\tp{ext}(D(b))$ and is path-connected.
\end{claim}

\begin{proof}
Lemma \ref{l:Krein-Milman result} directly implies that $D(b) \cap Q(\mathscr{A})$ contains $\operatorname{ext}(D(b))$. The path-connectedness of $D(b) \cap Q(\mathscr{A})$ follows from Lemma \ref{l:Lemma on Euclidean unit sphere}, the path-connectedness of $\ker(I-T_b) \cap \mathbb{S}_H$ and the continuity of $Q$.
\end{proof}

\begin{claim}
For every $b \in \mathbb{S}_{X_Q}$, the convex set $D(b)$ has finite affine dimension.
\end{claim}

\begin{proof}
Since $D(b) \cap Q(\mathscr{A})$ contains $\operatorname{ext}(D(b))$, it suffices to show that $D(b) \cap Q(\mathscr{A})$ has finite affine dimension. By Lemma \ref{l:Lemma on Euclidean unit sphere}, $D(b) \cap Q(\mathscr{A}) \subset Q(\ker(I-T_b))$. Since $K_b$ is compact and self-adjoint, $\ker(I-T_b)$ is finite-dimensional, and so $Q(\ker(I-T_B))$ has finite affine dimension. Precisely, when $\{\omega_1,\ldots,\omega_n\}$ is an orthonormal basis of $\ker(I-T_b)$ and $Q \omega \in D(b) \cap Q(\mathscr{A})$, we may write $Q \omega \in \operatorname{span} \cup_{j,k=1}^n \{Q'_{\omega_j} \omega_k\}$ (when $\mathbb{K} = \mathbb{R}$) or $Q \omega \in \operatorname{span} \cup_{j,k=1}^n \{\mathbf{Re}(Q'_{\omega_j} \omega_k), \mathbf{Re}(Q'_{\omega_j} (i \omega_k))\}$ (when $\mathbb{K} = \mathbb{C}$). 
\end{proof}

\begin{claim} \label{Claim on equivalent topologies}
The norm and relative weak$^*$ topologies coincide in $Q (\mathscr{A})$.
\end{claim}

\begin{proof}
Suppose $\omega_j,\omega \in \mathscr{A}$ and $Q \omega_j \overset{*}{\rightharpoonup} Q \omega$. Seeking a contradiction, suppose that for a subsequence, $\liminf_{j \to \infty} \norm{Q \omega_j - Q \omega}_{X_Q^*} > 0$. By passing to a further subsequence, $\omega_j \rightharpoonup \gamma \in \mathbb{B}_H$. Now $Q \omega_j \overset{*}{\rightharpoonup} Q \gamma$ and thus $Q \gamma = Q \omega$. This implies that $1 = \norm{Q \gamma}_{X_Q^*} \le \norm{\gamma}_H \le 1$. Thus $\omega_j \rightharpoonup \gamma$ and $\norm{\omega_j}_H \to \norm{\gamma}_H$, giving $\norm{\omega_j-\gamma}_H \to 0$ and so $\norm{Q \omega_j - Q \gamma}_H \to 0$, and the latter yields the sought contradiction.
\end{proof}

\subsection{Claims (iv)--(vi)}
We recall and prove Claims (iv)--(vi) below. In this subsection, we assume that Assumptions \ref{Assumption on the bilinear operator}, \ref{Assumption on the bilinear operator 2} and \ref{Assumption 2} hold.

\begin{claim}
For every $b$ in a relatively open dense subset of $\mathbb S_{X_Q}$, $D(b) = \{Q\omega\}$ for some $\omega \in \mathscr{A}$ and $\norm{\cdot}_{X_Q}$ is Fr\'{e}chet differentiable at $b$.
\end{claim}

\begin{proof}
By Theorem \ref{Asplund-Lindenstrauss theorem}, $\norm{\cdot}_{X_Q}$ is Fr\'echet differentiable in a dense subset of $\mathbb{S}_{X_Q}$. Seeking a contradiction, suppose $b$ is a point of Fr\'echet differentiability but the points $b_j \to b$, $b_j \in \mathbb S_{X_Q}$, are not. Thus the subspaces $\tp{ker}(I-T_{b_j})$ are at least two-dimensional but, by Corollary \ref{c:Corollary on points of Frechet differentiability}, $\tp{dim}(\tp{ker}(I-T_b)) = 1$.

For every $j \in \N$ choose $\omega_j, \gamma_j \in \ker(I-T_b) \cap \SSS_H$ such that $\langle \omega_j, \gamma_j \rangle_H = 0$. Since $\mathbb B_H$ is weakly compact, we may assume that $\omega_j \rightharpoonup \omega$ and $\gamma_j \rightharpoonup \gamma$ in $H$. Then $Q \omega_j \overset{*}{\rightharpoonup} Q \omega$ and $Q \gamma_j \overset{*}{\rightharpoonup} Q \gamma$ in $X_Q^*$.

Since $b_j \to b$ in $X_Q$ and $Q \omega_j \overset{*}{\rightharpoonup} Q \omega$ and $Q \gamma_j \overset{*}{\rightharpoonup} Q \gamma$ in $X_Q^*$, we conclude that
\begin{align*}
& \langle Q \omega, b \rangle_{X_Q^*-X_Q} = \lim_{j \to \infty} \langle Q \omega_j, b_j \rangle_{X_Q^*-X_Q} = 1, \\
& \langle Q \gamma, b \rangle_{X_Q^*-X_Q} = \lim_{j \to \infty} \langle Q \gamma_j, b_j \rangle_{X_Q^*-X_Q} = 1.
\end{align*}
Since $D(b)$ is a singleton by assumption, we obtain $Q \omega = Q \gamma$. By Corollary \ref{c:Corollary on points of Frechet differentiability},
\begin{equation} \label{Equality up to rotation}
\gamma = c \omega \quad \text{for some } c \in \mathbb{S}_{\mathbb{K}}.
\end{equation}
Furthermore,
\[1 = \lim_{j \to \infty} \norm{\omega_j}_H^2 \ge \norm{\omega}_H^2 \ge \|Q \omega\|_{X_Q^*} \ge \langle Q \omega, b \rangle_{X_Q^*-X_Q} = 1.\]
We thus have $\omega_j \rightharpoonup \omega$ and $\|\omega_j\|_H \to \|\omega\|_H$, whereby $\|\omega_j-\omega\|_H \to 0$. Similarly, $\|\gamma_j-\gamma\|_H \to 0$. We conclude, via \eqref{Equality up to rotation}, that
\[0 = \langle \omega_j, \gamma_j \rangle \to \langle \omega, \gamma \rangle = \langle \omega, c \omega \rangle = \overline{c},\]
which yields a contradiction.
\end{proof}

\begin{remark} \label{r:Higher dimensions}
Note that a straightforward modification of the proof above proves the relative openness in $\SSS_{X_Q}$ of $\{b \in \SSS_{X_Q} \colon \dim(\ker(I-K_b)) \le n\}$ for every $n \in \N$.
\end{remark}

\begin{claim}
$D \colon \mathbb{S}_{X_Q} \to 2^{\mathbb{S}_{X_Q^*}}$ is a cusco map.
\end{claim}

\begin{proof}
The duality mapping $D \colon S_{X_Q} \to S_{X_Q^*}$ is point-to-compact since for every $b \in \mathbb{S}_{X_Q}$, the closed, bounded set $D(b)$ is contained in a finite-dimensional subspace of $X_Q^*$. We next intend to show that $D$ is norm-to-norm upper semicontinuous. For this, suppose $b \in \mathbb{S}_{X_Q}$ and $f_j \in \mathbb{S}_{X_Q^*}$ satisfy $\lim_{j \to \infty} \langle f_j,b \rangle_{X_Q^*-X_Q} = 1$. By Theorem \ref{t:GGS} it suffices to show that $f_j \to f \in D(b)$ for a subsequence.

Denote $\operatorname{dim}(\ker(I-T_b)) = n \in \N$. Then $\operatorname{dim}(\ker(I-T_{b_j})) \le n$ from some index on by Remark \ref{r:Higher dimensions}. Thus we may write every $f_j$ as a convex combination $f_j = \sum_{k=1}^N \lambda_j^k Q \omega_j^k$. By passing to a subsequence, for every $k \in \{1,\ldots,N\}$ we have $\lambda_j^k \to \lambda^k \in [0,1]$ and $\omega_j^k \rightharpoonup \omega^k \in \mathbb{B}_H$. In particular, $\sum_{k=1}^N \lambda^k = 1$.

Suppose now $\lambda^k > 0$. Then necessarily $\langle Q \omega_j^k, b \rangle_{X_Q^*-X_Q} \to 1$ and so $\|\omega_j^k\|_H \to 1$, giving $\|\omega_j^k-\omega^k\|_H \to 0$ and $\|Q \omega_j^k - Q \omega^k\|_{X_Q^*} \to 0$. We conclude that $f_j \to \sum_{k=1}^N \lambda^k Q \omega^k \in D(b)$.
\end{proof}

\begin{claim} \label{Claim on equivalent topologies2}
The norm and relative weak$^*$ topologies also coincide in $\textit{NA}_{\norm{\cdot}_{X_Q}}$.
\end{claim}

\begin{proof}
The result follows directly from Claim (v) and Theorem \ref{t:GGS}.
\end{proof}

\section{The Jacobian equation with $L^p$ data} \label{s:The Jacobian equation with Lp data}
Many of the ideas of this paper can be adapted to study the range of the operator $J \colon \dot{W}^{1,2p}(\R^2,\R^2) \to L^p(\R^2)$ when $1 < p < \infty$. In this section we outline such an approach and list those results whose proof extends to this new setting in a straightforward manner.

We can again use basic properties of the Beurling transform to write
\[\langle b, Q\omega \rangle_{L^{p'}-L^p}= \langle T_b \omega, \omega \rangle_{L^{(2p)'}-L^{2p}}, \quad
T_b \omega \defeq \overline{(\S b-b\S) \overline{\S \omega}}, \quad Q \omega \defeq \abs{\S \omega}^2 - \abs{\omega}^2.\]
Note that $\norm{T_b}_{L^{2p} \to L^{(2p)'}} \lesssim_p \norm{b}_{L^{p'}}$ for all $b \in L^{p'}(\C)$ by a simple application of H\"older's inequality and the boundedness of $\S \colon L^p(\C,\C) \to L^p(\C,\C)$. The lower bound estimate
\begin{equation} \label{e:Lower bound estimate for Lp}
\norm{T_b}_{L^{2p} \to L^{(2p)'}} \gtrsim_p \norm{b}_{L^{p'}},
\end{equation}
however, is far from trivial and cannot be proved by simply following the proof of estimates \eqref{Two-sided estimate}.

Hyt\"onen proved \eqref{e:Lower bound estimate for Lp} rather recently in~\cite{Hyt21}. More generally, on $\R^n$, Hyt\"onen showed that the commutator of a degenerate Calder\'on-Zygmund operator with $b \in L^1_{loc}(\R^n)$ defines a bounded operator from $L^{q_1}(\R^n)$ into $L^{q_2}(\R^n)$, $q_1 > q_2 > 1$, if and only if $b = a+c$ with $a \in L^r(\R^n)$, $1/r=1/{q_2}-1/{q_1}$ and $c \in \R$. In notable contrast to the case of $[b,T] \colon L^2(\R^n) \to L^2(\R^n)$, the possible cancellations of $b$ do not play an important role; recall, for instance, that there exist $b \notin \tp{BMO}(\R^n)$ such that $\abs{b} \in \tp{BMO}(\R^n)$.

As a corollary of \eqref{e:Lower bound estimate for Lp}, Hyt\"onen obtained an analogue of \eqref{e:Jacobian decomposition}: for every $f \in L^p(\C)$ there exist $u_j \in \dot{W}^{1,2p}(\C,\C)$ such that $f = \sum_{j=1}^\infty J u_j$ and $\sum_{j=1}^\infty \norm{D u_j}_{L^{2p}}^2 \lesssim \norm{f}_{L^p}$~\cite{Hyt21}. Hyt\"onen also showed an analogous result in higher dimensions by different methods.

In order to adapt the methods of this paper to the $L^p$ case, another crucial ingredient is the compactness of $[b,T] \colon L^{2p}(\R^2,\R^2) \to L^{(2p)'}(\R^2,\R^2)$ for all $b \in L^{p'}(\R^2)$. Hyt\"onen \& al. recently proved this property for a large class of degenerate Calder\'on-Zygmund kernels in~\cite{HLTY22}. We formulate their two main results.
\begin{theorem}
Let $1 < q_2 < q_1 < \infty$ and $1/r=1/{q_2}-1/{q_1}$. Suppose $T$ is a non-degenerate Calder\'on-Zygmund operator which satisfies one of the following two:
\begin{enumerate}
\item condition (i) of Definition \ref{d:Non-degenerate kernels} with the Dini condition $\int_0^1 t^{-1} \omega(t) \tp{d} t < \infty$,

\item condition (ii) of Definition \ref{d:Non-degenerate kernels} with $\Omega \in L^\nu(S^{n-1})$ for some $\nu \in (1,\infty]$.
\end{enumerate}
Then $[b,T]$ is compact from $L^{q_1}(\R^n)$ to $L^{q_2}(\R^n)$ for all $b \in L^r(\R^n)$.
\end{theorem}

 In the next subsection we put the operator $Q = \abs{\S\cdot}^2-\abs{\cdot}^2 \colon L^{2p}(\C,\C) \to L^p(\C)$ into a general Banach-space geometric framework.

\subsection{Mathematical setting}
We assume that $X$ is a separable, reflexive Banach space and that $Y$ is a separable, reflexive, smooth, strictly convex Banach space with the Kadec-Klee property. Assumptions \ref{Assumption on the bilinear operator}--\ref{Assumption on the bilinear operator 2}, Definition \ref{Definition of the quadratic quantity} and Proposition \ref{p:Corollary on modified operator} have natural analogues:

\begin{assumption} \label{Assumption* on the bilinear operator}
A bilinear mapping
\[(b,f) \mapsto T_b f \colon X \times Y \to Y^*\]
satisfies the following conditions:

\renewcommand{\labelenumi}{(\roman{enumi})}
\begin{enumerate}
\item $c \norm{b}_X \le \norm{T_b}_{Y \to Y^*} \le C \|b\|_X$ for every $b \in X$,

\item $T_b$ is compact for every $b \in X$.
\end{enumerate}
\end{assumption}

\begin{assumption} \label{Assumption on the bilinear operator 2*}
The following conditions hold for every $b \in X$:
\begin{enumerate}
\item $T_b^* = T_b$.

\item $\norm{T_b}_{Y \to Y^*} = \sup_{\norm{\omega}_Y = 1} \langle T_b \omega, \omega \rangle_{Y^*-Y}$.
\end{enumerate}
\end{assumption}

\begin{definition} \label{d:Definition of Q*}
Suppose Assumptions \ref{Assumption* on the bilinear operator}--\ref{Assumption on the bilinear operator 2*} hold. Define the norm-to-norm and weak-to-weak sequentially continuous map $Q \colon Y \to X^*$ by
\[\langle b, Qf \rangle_{X-X^*} \defeq \langle T_b f, f \rangle_{Y^*-Y}.\]
\end{definition}

\begin{proposition} \label{p:modified Tb in Lp case}
If $T$ satisfies Assumption \ref{Assumption* on the bilinear operator}, then $\tilde{T}_b \colon Y \times Y \to Y^* \times Y^*$, $\tilde{T}_b(\omega,\gamma) \defeq (T_b^* \gamma,T_b \omega)$, satisfies Assumptions \ref{Assumption* on the bilinear operator}--\ref{Assumption on the bilinear operator 2*}.
\end{proposition}

In Proposition \ref{p:modified Tb in Lp case} we have endowed $Y \times Y$ with the norm $\norm{(\omega,\gamma)}_{Y \times Y} \defeq (\norm{\omega}_Y^2+\norm{\gamma}_Y^2)^{1/2}$ and similarly for $Y^* \times Y^*$, so that $Y \times Y$ and $Y^* \times Y^*$ are smooth and strictly convex since $Y$ and $Y^*$ are. Denoting $D(\omega) = \{\omega^*\}$ for all $\omega \in \SSS_Y$ we have $D(\omega,\gamma) = \{(\omega^*,\gamma^*)\}$ for all $\omega,\gamma \in \SSS_Y$.

\begin{remark}
Assumption \ref{Assumption on the bilinear operator 2*} reveals an important difference to the case of $\H^1$ data. In the case of $J \colon \dot{W}^{1,2}(\C,\C) \to \H^1(\C)$, the associated linear operators $T_b \colon L^2(\C,\C) \to L^2(\C,\C)$ are self-adjoint and therefore the operator norm and numerical radius of $T_b$ coincide, $\norm{T_b}_{L^2\to L^2} = \sup_{\norm{\omega}_{L^2}=1} \abs{\langle T_b f, f \rangle}$. The standard proof (see e.g.~\cite[p. 34]{Con90}) employs the fact that as a Hilbert space, $L^2(\C,\C)$ satisfies the parallelogram law $\norm{\omega+\gamma}_{L^2}^2+\norm{\omega-\gamma}_{L^2}^2 = 2 (\norm{\omega}_{L^2}^2+\norm{\gamma}_{L^2}^2)$. The analogue $\norm{T_b}_{L^{2p} \to L^{(2p)'}} = \sup_{\norm{\omega}_{L^{2p}}=1} \abs{\langle T_b \omega,\omega \rangle_{L^{(2p)'}-L^{2p}}}$ seems to fail (even though the left and right hand sides are comparable), owing to the fact that the inequality $\norm{\omega+\gamma}_{L^{2p}}^2 + \norm{\omega+\gamma}_{L^{(2p)'}}^2 \leq 2 (\norm{\omega}_{L^{2p}}^2+\norm{\gamma}_{L^{2p}}^2)$ fails (see~\cite[Theorem 2]{Kos79}).

In order to study the planar Jacobian via our operator theoretic framework, one then needs to make the concession that $Q f \defeq \abs{\S f}^2 - \abs{f}^2$ is replaced by $\tilde{Q}(f,g) = Q'_f g = J(u_1,v_2) + J(v_1,u_2)$, where $u = u_1 + i u_2 = \mathcal{C} f$ and $v = v_1 + iv_2 = \mathcal{C} g$. It is of course equivalent to studying the range of $(u,v) \mapsto Ju + Jv \colon \dot{W}^{1,2p}(\C,\C)^2 \to L^p(\C)$.
\end{remark}
  
We define norms on $X$ and $X^*$ by $\norm{b}_{X_Q} \defeq \sup_{\norm{\omega}_X=1} \langle b, Q \omega \rangle_{X-X^*} = \norm{T_b}_{Y \to Y^*}$ and $\norm{f}_{X_Q^*} \defeq \sup_{\norm{b}_{X_Q}=1} \langle f, b \rangle_{X^*-X}$. We also denote $\A \defeq \{\omega \in \SSS_Y \colon Q \omega \in \SSS_{X_Q^*}\}$.

\section{Results}
In the current setting, the Lagrange multiplier condition
\[\left. \frac{d}{d\epsilon} (\langle b, Q(\omega + \epsilon \varphi) \rangle_{X_Q-X_Q^*} - \norm{\omega + \epsilon \varphi}_Y^2) \right|_{\epsilon = 0} = 0 \qquad \text{for every } \varphi \in Y\]
is written concisely as $T_b \omega = \omega^*$, where $\omega^*$ is the unique element of $Y^*$ such that $\langle \omega, \omega^* \rangle_{Y-Y^*} = \norm{\omega}_Y \norm{\omega^*}_{Y^*}$. (In particular, $D(\omega) = \{\omega^*\}$ for all $\omega \in \SSS_Y$.) This is seen by following the proof of Proposition \ref{p:Proposition on equivalent characterizations of the Lagrange multiplier condition}.

We list results of this paper that allow a straightforward adaptation to the current setting: Theorem \ref{t:GKL open mapping theorem}; Corollaries \ref{c:Corollary on James boundaries}, \ref{c:Norm corollary}, \ref{c:Collection of conditions} and \ref{c:Corollary on points of Frechet differentiability}; Propositions \ref{p:Corollary on modified operator}, \ref{p:Existence of minimum norm solutions}, \ref{p:Weak continuity proposition}, \ref{p:Proposition on existence of Lagrange multipliers}, \ref{p:Proposition on equivalent characterizations of the Lagrange multiplier condition} and \ref{p:Having a Lagrange multiplier}; Lemmas \ref{l:Lemma on the self-adjoint modification}, \ref{l:Isometric isomorphism} and \ref{l:Krein-Milman result}.

Theorem \ref{t:Theorem on duality mapping} (i) has at least the weaker variant that $D \colon \SSS_{X_Q} \to \SSS_{X_Q^*}$ is point-to-compact. It is natural to ask whether, again, $\tp{co}(\tp{ext}(D(b))) = D(b)$ for all $b \in \SSS_{X_Q}$. Since $\tp{ext}(\B_{X_Q^*}) \subset Q\A$, we could then write each $f \in \SSS_{X_Q^*}$ as a convex combination of elements $Q \omega$, $\omega \in \A$. By a standard Baire category argument, one would then obtain an upper bound for the number of terms in these sums~\cite{GKL20}:

\begin{proposition} \label{p:Finitary decomposition proposition}
Suppose $\tp{co}(\tp{ext}(D(b))) = D(b)$ for all $b \in \SSS_{X_Q}$. Then there exists $N \in \N$ such that every $f \in X_Q^*$ can be written as
\[f = \sum_{j=1}^N Q \omega_j, \qquad \sum_{j=1}^N \norm{\omega_j}_{Y}^2 \lesssim \norm{f}_{X_Q^*}.\]
\end{proposition}

\begin{appendix}

\section{A proof of \eqref{Product of Hardy spaces}} \label{s:Appendix}
While formula \eqref{Product of Hardy spaces} is well-known and classical, the author has been unable to find a proof in the literature, and therefore one is sketched in this appendix. Note that \eqref{Product of Hardy spaces} is equivalent to the following proposition.

\begin{proposition} \label{p:Classical surjectivity result}
Every $f \in \mathcal{H}^1(\R)$ can be written as
\[f = \omega \gamma - H \omega H \gamma\]
for some $\omega, \gamma \in L^2(\R)$. However, there exists $f \in \mathcal{H}^1(\R)$ that cannot be written as $f = \omega^2 - (H \omega)^2$ for any $\omega \in L^2(\R)$.
\end{proposition}

The relevant background can be found e.g. in~\cite{Dur70} or~\cite{Mas09}. Whenever $0 < p < \infty$, we denote
\[\norm{U}_p \defeq \sup_{0 < y < \infty} \left( \int_{-\infty}^\infty \abs{U(x+iy)}^p dx \right)^\frac{1}{p}.\]
The analytic Hardy space $\mathcal{H}^p(\C_+)$ consists of analytic functions $U \colon \C_+ \to \C$ with $\norm{U}_p < \infty$.

\begin{proof}[Proof of Proposition \ref{p:Classical surjectivity result}]
Let $f \in \mathcal{H}^1(\R)$; then the Hilbert transform $H f \in L^1(\R)$. We extend $f+ i H f$ analytically into $\C_+$ by using the Poisson kernel $P_y$ and the conjugate Poisson kernel $Q_y$,
\[P_y(x) \defeq \frac{1}{\pi} \frac{y}{x^2+y^2}, \quad Q_y(x) \defeq \frac{1}{\pi} \frac{x}{x^2+y^2}.\]
We denote
\[U_1(x+iy) + i U_2(x+iy) \defeq f * P_y(x) + i \, f * Q_y(x);\]
then $U = U_1 + i U_2$ belongs to $\mathcal{H}^1(\C_+)$ and, furthermore, $\lim_{y \searrow 0} \norm{U_1(\cdot,y) - f}_{L^1} = 0$ and $\lim_{y \searrow 0} \norm{U_2(\cdot,y) - H f}_{L^1} = 0$.

We form the Blaschke product $B$ of $U$ and write $U = B \Phi$, where $\Phi \in \mathcal{H}^1(\C_+)$ is zero-free. Since $\C_+$ is simply connected, $\Phi$ has an analytic square root $\Phi^{1/2}$. Set
\[V \defeq B \Phi^{1/2} \in \mathcal{H}^2(\C_+) \quad \text{and} \quad W \defeq \Phi^{1/2} \in \mathcal{H}^2(\C_+).\]
Then there exist $\omega, \gamma \in L^2(\R)$ such that
\[V_1(\cdot + iy) + i V_2(\cdot + iy) \to \omega + i H \omega,\]
\[W_1(\cdot + iy) + i W_2(\cdot + iy) \to \gamma + i H \gamma\]
in $L^2(\R)$ as $y \searrow 0$. As a consequence,
\[\mathbf{Re} [VW] = V_1 W_1 - V_2 W_2 \to \omega \gamma - H \omega H \gamma\]
in $L^1(\R)$ as $y \searrow 0$. On the other hand,
\[\mathbf{Re} [VW] = \mathbf{Re} \, U \to f\]
in $L^1(\R)$ as $y \searrow 0$. Thus $f = w \gamma - H \omega H \gamma$.

\vspace{0.3cm}
We then find $f \in \mathcal{H}^1(\R)$ that cannot be written as $f = \omega^2 - (H \omega)^2$ for any $\omega \in L^2(\R)$. Choose $U \in \mathcal{H}^1(\C_+)$ that has at least one zero of odd order. (Given $W \in \mathcal{H}^1(\C_+)$ one may set $U(z) = [(z-z_0)/(z-\bar{z_0})] W(z)$, where $z_0 \in \C_+$ and $W(z_0) \neq 0$.) Then $U$ is not the square of any analytic function. We write $f \defeq \lim_{y \searrow 0} U_1 \in \mathcal{H}^1(\R)$ where the limit holds in $L^1(\R)$.

Seeking a contradiction, assume that $f = \omega^2 - (H \omega)^2$ for some $\omega \in L^2(\R)$. Then there exists $V = V_1 + i V_2 \in \mathcal{H}^2(\C_+)$ such that $\lim_{y \searrow 0} \norm{V_1(\cdot+iy) - \omega}_{L^2} = 0$. Now $V^2 \in \mathcal{H}^1(\C_+)$ and
\[\Re[V^2] = V_1^2 - V_2^2 \to \omega^2 - (H \omega)^2 = f\]
in $L^1(\R)$ as $y \searrow 0$. Thus $U-V^2 \in \mathcal{H}^1(\C_+)$ has vanishing boundary values at $y = 0$, and so $U = V^2$, which yields a contradiction.
\end{proof}

\end{appendix}

\bibliography{Jacobianbibliography}
\bibliographystyle{amsplain}

\end{document}